\documentclass{amsart}
\usepackage[utf8]{inputenc}
\usepackage{amsfonts}
\usepackage{amssymb}
\usepackage{amsmath}
\usepackage{amsthm}
\usepackage{amscd}
\usepackage{amstext}
\usepackage{appendix}
\usepackage{bm}
\usepackage{caption}
\usepackage{comment}
\usepackage{enumitem}
\usepackage{fancyhdr}
\usepackage{hyperref}
\usepackage{graphicx}
\usepackage{lipsum}
\usepackage{mathabx}
\usepackage{mathrsfs}
\usepackage{mathtools}
\usepackage{romannum}
\usepackage[section]{placeins}
\usepackage{stmaryrd}
\usepackage[switch, modulo]{lineno}
\usepackage{tikz}
\usepackage{xcolor}
\hypersetup{
    colorlinks,
    linkcolor={black},
    citecolor={black},
    urlcolor={black}
}

\theoremstyle{plain}
    \newtheorem{theorem}{Theorem}[section]
    \newtheorem{lem}[theorem]{Lemma}
    \newtheorem{prop}[theorem]{Proposition}
    \newtheorem{cor}[theorem]{Corollary}
    \newtheorem{rem}[theorem]{Remark}
\theoremstyle{definition}
    \newtheorem{defn}[theorem]{Definition}

\renewcommand{\thepage}{\roman{page}}
\title{Characterizations of Type-$III$ Bernoulli Scheme}
\author{Tianyi Zhou}
\date{}
        
\begin{document}

\clearpage
\thispagestyle{empty}

\begin{abstract}
In this paper, we will prove alternate conditions for a type-$III$ Bernoulli scheme to be of type-$III_0$, type-$III_{\lambda}$ and type-$III_1$, and then conclude an alternate definition of the asymptotic ratio set of a type-$III$ ITPFI factor.
\end{abstract}

\maketitle

\vspace{0.5cm}
\begin{center}
\large\bfseries{Introduction}
\end{center}
\vspace{0.5cm}

Bernoulli scheme is one of popular examples of a non-singular and ergodic group action defined on a standard probability space. As a result, it can be classified by its ratio set (see \cite{18}). The relation between a Bernoulli scheme and its associated ITPFI factor was studied in \cite{2}, \cite{26} and then \cite{73}, and \cite{73} points out that when the Bernoulli scheme is of type-$III$, its ratio set coincides with the asymptotic ratio set of the ITPFI factor associated with the given Bernoulli scheme. In \cite{2}, a precise description of a type-$I$, type-$II_1$ Bernoulli scheme is proven, along with a necessary condition for a Bernoulli scheme to be of type-$III$. Results in \cite{2} were later used in \cite{26} to describe type-$I$, type-$II_1$ ITPFI factors, and to provide a necessary condition for an ITPFI factor to be of type-$III$. In \cite{80}, a precise description of a type-$III$ ITPFI factor is also proven. Also \cite{4} and \cite{7} provide necessary conditions respectively for a two-point Bernoulli scheme to be of type-$III_1$ where the method in \cite{4} is statistical while the one in \cite{7} inspires techniques that would be used in proving our main theorems. Both methods will need observations on clustered points of $\left\{ \mu_n(1) \slash \mu_n(0) \right\}$, the set of ratio of weights of each $\mu_n$. In the context of a two-point Bernoulli scheme, certain summability conditions will determine the type of the measure, as shown in \cite{7} and \cite{15}.\\

The goal of this paper is to prove alternate conditions that describe the set of clustered points of ratio of weights of each $\mu_n$ (as in \cite{7}) to describe each type of a type-$III$ Bernoulli scheme. Then we will conclude an alternate definition of the asymptotic ratio set of a type-$III$ ITPFI factor. Throughout this paper in a measure space, all equations and definitions are assumed $\bmod\,0$. In the preliminary section, definitions we need and some of the historical results will be covered. In \textbf{Section 2.1}, we will first provide a description of type-$III_0$, type-$III_{\lambda}$ and type-$III_1$ Bernoulli scheme based on the local behavior of their Radon-Nikodym cocycles. In the first part of \textbf{Section 2.2}, main theorems are proven for type-$III$ Bernoulli scheme of unbounded type (\cite{40}). The main theorems for bounded type Bernoulli scheme (see \cite{41}) are proven in \textbf{Section 2.2}.

\setcounter{page}{1}
\renewcommand{\thepage}{\arabic{page}}
\fancyhf[rh]{\arabic{page}}

\section{Preliminaries}

\subsection{ITPFI factors}

\begin{defn}[{\cite[Definition 2.3]{26}}]\label{Definition 1.1}

Given a countable family of Hilbert spaces $\big( H_n\big)_{n\in \mathbb{N}}$ where the dimension of each $H_n$ is at most countably infinite. Suppose $\big( v_n \big)_{n\in \mathbb{N}}$ is a family of vectors such that $v_n\in H_n$ for each $n\in\mathbb{N}$ and $\prod_{n\in \mathbb{N}} \|v_n\|$ converges to a positive number. Then the \textbf{incomplete tensor product space} (\textbf{ITPS}) of $\big( H_n \big)_{n\in \mathbb{N}}$ is denoted by:

$$
H = \bigotimes_{n\in \mathbb{N}} \big(H_n, v_n\big)
$$
For any $(u_n)_{n\in \mathbb{N}}$ with $u_n\in H_n$, if $\prod_{n\in \mathbb{N}} \|u_n\|$ converges to a positive number, by definition we have:

$$
u = \bigotimes_{n\in \mathbb{N}} u_n\in \bigotimes_{n\in \mathbb{N}} H_n
$$
Without confusion we simply say $u\in H$. Similarly, we have $v = \bigotimes_{n\in \mathbb{N}} v_n\in H$. For any subset $J\subseteq \mathbb{N}$, we define:

$$
H(J) = \bigotimes_{n\in J}H_n, \hspace{1cm} v(J) = \bigotimes_{n\in J} v_n
$$

\end{defn}

\begin{defn}[{\cite[Definition 2.4, 2.5]{26}}]\label{Definition 1.2}

In the set-up of \textbf{Definition 1.1}, let $H = \bigotimes_{n\in \mathbb{N}} \big( H_n, v_n \big)$ be an IPTS of $\big( H_n \big)_{n\in \mathbb{N}}$. For any $k\in\mathbb{N}$ and $S\in B(H_k)$, define:

$$
\pi_k: B(H_k)\rightarrow B(H), \hspace{0.3cm} S\mapsto \left( \bigotimes_{\substack{n\in \mathbb{N} \\ n\neq k}} \operatorname{id}_k \right) \otimes S
$$
Given a sequence of von Neumann algebra $\big( \mathcal{M}_n \big)_{n\in \mathbb{N}}$ where $\mathcal{M}_n \subseteq B(H_n)$ for each $n\in\mathbb{N}$, define:

$$
\bigotimes_{n\in \mathbb{N}} \mathcal{M}_n = \big\{ \pi_n(\mathcal{M}_n): n\in \mathbb{N} \big\}^{''}
$$
For any subset $J\subseteq \mathbb{N}$, we define:

$$
\mathcal{M}(J) = \bigotimes_{n\in J} \mathcal{M}_n
$$
When all $\mathcal{M}_n$ are factors, we use $\mathcal{M} = R\big( H_n, \mathcal{M}_n, v_n \big)$ to denote the factor $\bigotimes_{n\in \mathbb{N}} \mathcal{M}_n$ endowed with $v$. When $H$ is fixed, we will simply use $\mathcal{M} = R\big( \mathcal{M}_n, v_n \big)$.
    
\end{defn}

\begin{defn}[{\cite[Definition 2.6]{26}}]

Given an IPTS $H = \bigotimes_{n\in \mathbb{N}} \big( H_n, v_n \big)$, assume that the dimension of all $H_n$ are at least $2$. Then an \textbf{ITPFI factor} is an factor that is isomorphic to $R\big( \mathcal{M}_n, v_n \big)$ where each $\mathcal{M}_n$ is a type-$I_{m_n}$ factor where $m_n\in\mathbb{N}$ and $m_n\geq 2$. In particular, if for some $k\in \mathbb{N}$, $m_n=k$ for all $n\in\mathbb{N}$, we call $R\big(\mathcal{M}_n, v_n \big)$ an $\textbf{\textup{ITPFI}}_k$ \textbf{factor}.
    
\end{defn}

\begin{defn}[\cite{26}]\label{Definition 1.4}

Given an ITPFI factor $\mathcal{M} = R\big( \mathcal{M}_n, v_n \big)$, for each $n\in\mathbb{N}$, observe that the following linear functional is a normal state defined on $\mathcal{M}_n$:

$$
\rho_n: \mathcal{M}_n \rightarrow \mathbb{C}, \hspace{0.3cm} S\mapsto \langle\, Sv_n,v_n\,\rangle
$$
and for an arbitrary finite subset $J\subseteq \mathbb{N}$, define:

$$
\rho_J = \mathcal{M}(J) \rightarrow \mathbb{C}, \hspace{0.3cm} \bigotimes_{n\in J} S_n \mapsto \prod_{n\in J} \langle\, Sv_n, v_n\,\rangle
$$
Hence for each $n\in\mathbb{N}$, there exists a positive trace-class operator $t_n\in B(H_n)$ such that $\rho_n(S) = \operatorname{Tr}\big( t_nS \big)$. For each $n\in\mathbb{N}$, we let $\operatorname{Sp}\big( v(n) \slash \mathcal{M}(n) \big)$ to denote the eigenvalues of $\rho_n$. Suppose:

$$
\operatorname{Sp}\big( v(n)\slash \mathcal{M}(n) \big) = \big\{ \lambda(n, i_n) \big\}_{1\leq i_n \leq k_n}
$$
and $\lambda_{n, 1} \geq \lambda_{n, 2} \geq \cdots$. Similarly, for any finite subset $J\subseteq \mathbb{N}$, we use $\operatorname{Sp}\big( v(J) \slash \mathcal{M}(J) \big)$ to denote the spectrum of the positive trace-class operator (in $B\big( H(J) \big)$) that corresponds to $\rho_J$. If $\lambda \in \operatorname{Sp}\big( v(J) \slash \mathcal{M}(J) \big)$, then clearly:

$$
\lambda = \prod_{n\in J}\lambda_{n, i_n}
$$
where, for each $n\in J$, $1\leq i_n\leq k_n$ and $\lambda_{n, i_n} \in \operatorname{Sp}\big( v(n)\slash \mathcal{M}(n) \big)$.
    
\end{defn}

\begin{defn}[{\cite[Definition 3.2]{26}}]\label{Definition 1.6}

Given an ITPFI factor $\mathcal{M} = R\big( \mathcal{M}_n, v_n \big)$, a finite subset $I\subseteq \mathbb{N}$ and $K\subseteq \operatorname{Sp}\big( v(I) \slash \mathcal{M}(I) \big)$, define $\lambda(K)$ to be the sum of elements in $K$. The \textbf{asymptotic ratio set} of $\mathcal{M}$, denoted by $r_{\infty}(\mathcal{M}, v)$ is the set of all non-negative real numbers $x$ such that there exists $\big( I_n \big)_{n\in \mathbb{N}}$ an infinite sequence of mutually disjoint finite subsets of $\mathbb{N}$ and a family of mutually disjoint subsets $\big( K_n^1, K_n^2 \big)_{n\in \mathbb{N}}$ such that for each $n\in\mathbb{N}$, $K_n^1$, $K_n^2\subseteq I_n$ are two disjoint subsets of non-zero eigenvalues and there exists a bijection $\phi_n$ from $K_n^1$ to $K_n^2$ such that the sequence $\big\{ \lambda\big( K_n^1 \big) \big\}_{n\in \mathbb{N}}$ is not summable and:

$$
\lim_n \max\left\{ \left\vert\, x-\frac{\phi_n(\lambda)}{\lambda} \,\right\vert: \lambda\in K_n^1 \right\} = 0
$$
    
\end{defn}

\begin{rem}

According to {\cite[Corollary 2.9]{26}}, in the set-up of \textbf{Definition 1.3}, $R\big( \mathcal{M}_n, v_n \big)$ is still an ITPFI factor if countably many $\mathcal{M}_n$ are infinite type-$I$ factors. Therefore, we shall also consider the case where infinitely many $\mathcal{M}_n$ are infinite-dimensional.
    
\end{rem}

\begin{lem}[{\cite[Lemma 3.14]{26}}]\label{Lemma 1.7}

Given $\mathcal{M} = R\big( \mathcal{M}_n, v_n \big)$ an ITPFI factor, for each $n\in\mathbb{N}$ define $u_n = \dfrac{v_n}{\|v_n\|}$ and define $\mathcal{M}^u = R\big( \mathcal{M}_n, u_n \big)$. We then have $\mathcal{M}\cong \mathcal{M}^u$ and $r_{\infty} \big( \mathcal{M}, v \big) = r_{\infty}\big( \mathcal{M}^u, u \big)$.
    
\end{lem}

\begin{prop}[{\cite[Lemma 3.6, 3.7]{26}}]\label{Proposition 1.8}

In the set-up of \textbf{Definition 1.5}, $r_{\infty}(\mathcal{M}, v)$ is a closed subset in $[0, \infty)$ and $r_{\infty}(\mathcal{M}, v) \backslash \{0\}$ is a multiplicative subgroup of $\mathbb{R}^*_+$.
    
\end{prop}

\begin{theorem}[{\cite[Theorem 3.9]{26}}]\label{Theorem 1.9}

Given $\mathcal{M} = R\big( \mathcal{M}_n, v_n \big)$ a ITPFI factor, $r_{\infty} (\mathcal{M}, v)$ must be either $\{1\}$, $\{0\}$, $\{0, 1\}$, $[0, \infty)$ or $\{0\} \cup \big( \lambda^k \big)_{k\in \mathbb{Z}}$ for some $\lambda\in (0, 1)$.
    
\end{theorem}

\begin{defn}[{\cite[Definition 3.10]{26}}]\label{Definition 1.11}

Given $\lambda\in [0, 1]$, we use $R_{\lambda}$ to denote the ITPFT factor $R\big( \mathcal{M}_n, v_n \big)$ such that for all $n\in\mathbb{N}$, $\mathcal{M}_n = M_2(\mathbb{C})$ and there exists two non-negative real numbers $\lambda_1, \lambda_2$ (at least one of them is non-zero) such that $\operatorname{Sp}\big( v(n) \slash \mathcal{M}(n) \big) = \{\lambda_1, \lambda_2\}$ for all $n\in\mathbb{N}$ and $\lambda \lambda_1 = \lambda_2$.
    
\end{defn}

\begin{theorem}[{\cite[Theorem 5.9]{26}}]\label{Theorem 2.14}

Given $\mathcal{M}$ an ITPFI factor, $r_{\infty}(\mathcal{M})$ is an \textbf{algebraic invariance}, namely for each $x\in [0, 1]$, $x\in r_{\infty}(\mathcal{M})$ if and only if $\mathcal{M} \cong \mathcal{M} \otimes R_x$. Therefore we can write rewrite the asymptotic ratio set of $\mathcal{M}$ as the following:

$$
r_{\infty}(\mathcal{M}, v) = \big\{ x\in [0, 1]: \mathcal{M} \cong \mathcal{M} \otimes R_x \big\} \cup \big\{y\in (1, \infty): \mathcal{M} \cong \mathcal{M} \otimes R_{y^{-1}} \big\}
$$
    
\end{theorem}

\subsection{Ratio set of a non-singular group action}

\begin{defn}

Given an ergodic non-singular group action $G\curvearrowright (X, \mathcal{B}, \mu)$ on a standard measure space, a \textbf{real-valued cocycle} is a function $c: G\times X \rightarrow \mathbb{R}$ that satisfies:

$$
c(g_1 g_2, x) = c(g_1, x) c(g_2, g_1^{-1} x)
$$
for all $x\in X$ and $g_1, g_2\in G$. In particular, the \textbf{Radon-Nikodym cocycle} of the given $G$-action is defined as the following:

$$
D: G\times X \rightarrow \mathbb{R}, \hspace{0.3cm} (g, x)\mapsto \frac{d(g\mu)}{d\mu}(x)
$$
and we can see for any $g_1, g_2\in G$ and $\mu$-almost every $x\in X$:

$$
\frac{d(g_2\mu)}{d\mu}\big( g_1^{-1}x \big) = \frac{d(g_1g_2)\mu}{dg_1 \mu}(x) \hspace{0.6cm} \Longrightarrow \hspace{0.6cm}
\frac{d(g_1g_2)\mu}{d\mu}(x) = \frac{dg_1\mu}{d\mu}(x) \, \frac{dg_2\mu}{d\mu}\big( g_1^{-1}x \big)
$$
If we let $R$ denote the orbit equivalence of the given group action (as a subset of $X\times X$), we can define the Radon-Nikodym cocycle in the following way: for any $g\in G$ and $x\in X$:

$$
D(x, gx) = \frac{d\big( g^{-1}\mu \big)}{d\mu}(x)
$$
which implies for any $g_1, g_2\in G$ and $\mu$-almost every $x\in X$:

$$
D(x, g_1x) D(g_1x, g_2g_1x) =  \frac{d\big( g_1^{-1}\mu \big)}{d\mu}(x)  \frac{d\big( g_2^{-1}\mu \big)}{d\mu}(g_1 x) = \frac{d\big( g_1^{-1}g_2^{-1} \mu \big)}{d\mu}(x) = D(x, g_2g_1x)
$$
    
\end{defn}

\begin{defn}\label{Definition 1.14}

Given an ergodic non-singular group action $G\curvearrowright (X, \mathcal{B}, \mu)$ on a standard measure space and a real-valued cocycle $c$, the \textbf{essential range} of $c$ is the set of all real numbers $r$ such that for any $\epsilon\in (0, 1)$ and any measurable subset $B\subseteq X$ with $\mu(B)>0$, there exist $B_1$, $B_2$ two non-negligible subsets of $B$ and $g\in G$ such that $gB_1 \subseteq B_2$ and $\big\vert\, c(g, x) - r \,\big\vert < \epsilon$ for all $x\in B_1$. The essential range of the Radon-Nikodym cocycle is called the \textbf{ratio set} of the given group action, and is denoted by $r_G(\mu)$.

\end{defn}

\begin{theorem}[{\cite[Theorem 3.9]{18}}]\label{Theorem 1.13}

In the set-up of \textbf{Definition \ref{Definition 1.14}}, $r_G(\mu)$ the ratio set of a given non-singular group action $G\curvearrowright (X, \mathcal{B}, \mu)$ satisfies:

\begin{itemize}
    \item $r_G(\mu)$ is a closed subset of $[0, \infty)$.
    \item $r_G(\mu) \backslash \{0\}$ is a multiplicative subgroup of $(0, \infty)$.
    \item $r_G(\mu)=r_G(\sigma)$ for any measures $\sigma$ that is equivalent to $\mu$.
\end{itemize}
    
\end{theorem}

\begin{defn}

Given an ergodic non-singular group action $G\curvearrowright \big(X, \mathcal{B}, \mu\big)$ on a standard measure space, $\mu$ is of:

\begin{itemize}

    \item type-$I$ if $\mu$ is equivalent to a purely atomic measure. When $\mu$ is of type-$I$, $\mu$ is of type-$I_n$ for some $n\in \mathbb{N}$ if $\mu$ only has $n$-many atoms; otherwise $\mu$ is of type-$I_{\infty}$.

    \item type-$II$ if $\mu$ is equivalent to a $\sigma$-finite $G$-invariant measure. When $\mu$ is of type-$II$, $\mu$ is of type-$II_1$ if $\mu$ is equivalent to a finite $G$-invariant measure; otherwise $\mu$ is of type-$II_{\infty}$ measure.

    \item type-$III$ if $\mu$ is not of type-$I$ or type-$II$. When $\mu$ is of type-$III_0$ if the ratio set of $\mu$ is $\{0, 1\}$; $\mu$ is of type-$III_{\lambda}$ for some $\lambda\in (0, 1)$ if the ratio set of $\mu$ is $\{0\}\cup \big\{ \lambda^n \big\}_{n\in \mathbb{Z}}$; $\mu$ is of type-$III_1$ if the ratio set of $\mu$ is $[0,\infty)$.
    
\end{itemize}
    
\end{defn}

\begin{lem}[{\cite[\textbf{Lemma 15}]{12}}]\label{Remark 1.21}

In the set-up of \textbf{Definition 1.15}:

\begin{enumerate}[label = (\arabic*)]

    \item if $\mu$ is of type-$III$, then $0$ is in the ratio set of $\mu$.
    \item $\mu$ is of type-$II$ if and only if $r_G(\mu) = \{1\}$.
    
\end{enumerate}
    
\end{lem}

\subsection{Maharam extension}

\begin{defn}[{\cite[Definition 3.1]{16}}]\label{Definition 1.20}

Given a non-singular group action $G\curvearrowright \big(X, \mathcal{B}, \mu\big)$, the \textbf{Maharam extension of} $G\curvearrowright (X, \mathcal{B}, \mu)$ is defined by:

$$
G\curvearrowright X \times \mathbb{R}: g\,\ast\, (x, t) = \big( g\,\cdot\,x, t+\log D (x, g\,\cdot\,x) \big) 
$$
where $D$ is the Radon-Nikodym cocycle of $G\curvearrowright \big(X, \mathcal{B}, \mu\big)$. In the Maharam extension, $\mathbb{R}$ is equipped with the finite Borel measure $\nu$ such that $d\nu = \operatorname{exp}\big( -\vert t\vert \big)dt$. The \textbf{associated flow of} $G\curvearrowright (X, \mathcal{B}, \mu)$ is the following $\mathbb{R}$-action on $\big( X\times \mathbb{R}, \mu\times \nu \big)$:

$$
\mathbb{R}\curvearrowright X\times  \mathbb{R}: r\,\cdot\, (x, t) = (x, r+t)
$$
One can check that associated flow of $G\curvearrowright X\times  \mathbb{R}$ commutes with the $G$-action of its Maharam extension. The $G$-action on the space of measurable functions on the Maharam extension is defined by the following: for each measurable function $F$ and $g\in G$:

$$
\forall\,(x, t)\in X\times\mathbb{R}, \hspace{0.5cm} \big(g\,\ast\,F \big)(x, t) = F\big( g\,\ast\, (x, t) \big) = F\big( g\,\cdot\,x, t+\log D(x, g\,\cdot\,x) \big)
$$
The $\mathbb{R}$-action on the space of measurable functions on the Maharam extension is defined by the following: for each measurable function $F$ and $r\in\mathbb{R}$:

$$
\forall\,(x, t)\in X\times\mathbb{R}, \hspace{0.5cm} \big(r\,\cdot\,F \big)(x, t) = F\big( r\,\cdot\, (x, t) \big) = F(x, t+r)
$$
From now on, we use $L^{\infty}\big( X\times \mathbb{R}, \mu \times \nu\big)^G$ to denote the space of $G$-invariant essentially bounded functions defined in the Maharam extension of $G\curvearrowright \big(X, \mathcal{B}, \mu \big)$.
    
\end{defn}

In the set-up of \textbf{Definition \ref{Definition 1.20}}, if we let $Z$ denote the space of ergodic components in the system and $\eta$ denote the image of $\mu$ under the quotient mapping, then $L^{\infty}\big( X\times \mathbb{R}, \mu \times \nu\big)^G$ can be identified as $L^{\infty}(Z, \eta)$. Equivalently the associated flow can be viewed as the $\mathbb{R}$-action on $L^{\infty}(Z, \eta)$. According to \cite{13}, the classification of $\mu$ can  then be described by the kernel of the associated flow as the following:

\begin{theorem}[\cite{13}]

In the set-up of \textbf{Definition \ref{Definition 1.20}}, the associated flow of $G\curvearrowright (X, \mathcal{B}, \mu)$ is:

\begin{enumerate}[label = (\arabic*)]

    \item conjugate with the translation action of $\mathbb{R}$ on itself if and only if $\mu$ is of type-$I$ or type-$II$.

    \item trivial if and only if $\mu$ is of type-$III_1$.

    \item conjugate with translation action $\mathbb{R}$ on $\mathbb{R}\slash \mathbb{Z} \log\lambda$ if and only if $\lambda$ is of type-$III_{\lambda}$.

    \item Every orbit of the associated flow (viewed as the the action $\mathbb{R}\curvearrowright (Z, \eta)$) is $\eta$-null if and only if $\mu$ is of type-$III_0$.
    
\end{enumerate}

\end{theorem}

\subsection{Bernoulli scheme}

\begin{defn}\label{Definition 1.16}

Given a sequence of countable spaces $(X_n)_{n\in \mathbb{N}}$, a \textbf{Bernoulli scheme} is the product space $X=\prod_{n\in \mathbb{N}}X_n$ endowed with a product measure $\mu=\otimes_{n\in \mathbb{N}} \mu_n$ where each $\mu_n$ is fully supported on $X_n$, the $\sigma$-algebra generated by cylinder sets in $X$ and the \textbf{synchronous tail relation} $R$ is defined by: any $\bm{x}$, $\bm{y}\in X$, we have $(\bm{x}, \bm{y})\in R$ if and only if there exists $N\in\mathbb{N}$ such that $x_n=y_n$ for all $n\geq N$. 
    
\end{defn}

\begin{defn}\label{Definition 1.17}

In the set-up of \textbf{Definition \ref{Definition 1.16}}, for each $n\in\mathbb{N}$, let $G_n$ be the permutation group of $X_n$ and $\bigtimes_{n\in \mathbb{N}}G_n$ be the group consisting the product group. For each $n\in\mathbb{N}$, let $e_n$ denote the identity of $G_n$. Then define:

$$
G = \left\{ \bm{g} = (g_n)_{n\in \mathbb{N}}\in \bigtimes_{n\in \mathbb{N}}G_n: \exists\,m\in\mathbb{N}\, \text{  such that  }\, \forall\, n\geq m,\, g_n=e_n \right\}
$$

\end{defn}

Clearly the synchronous tail relation coincides with the orbit equivalence relation associated with $G$ that acts on $X$ coordinate-wise. For each $g\in G$, we have:

$$
\bm{g}\mu = \bigotimes_{n\in \mathbb{N}} g_n\mu_n
$$
which implies that the $G$-action in \textbf{Definition \ref{Definition 1.17}} is ergodic and non-singular with respect to $\mu$. First note that the precise descriptions of type-$I$ and type-$II$ Bernoulli scheme are proven in \cite{2}. Then, with results in \cite{2}, descriptions of type-$I$ and type-$II_1$, type-$III$ (resp.) IFPTI factors are proven in \cite{26}, \cite{80} (resp.).

\begin{theorem}[{\cite[Theorem 1-(1)]{2}}]\label{Theorem 1.18}

In a Bernoulli scheme $G\curvearrowright (X, \mathcal{B}, \mu)$, $\mu$ is of type-$I$ if and only if:

\begin{equation}\label{e1}
\sum_{n\in\mathbb{N}}\left( 1-\max_{a\in X_n} \mu_n(a) \right) < \infty
\end{equation}
    
\end{theorem}

\begin{theorem}[{\cite[Theorem 1-(2)]{2}}]\label{Theorem 1.19}

In a Bernoulli scheme $G\curvearrowright \big(X, \mathcal{B}, \mu\big)$, $\mu$ is of type-$II_1$ if and only if all $X_n$ are finite, infinitely many $X_n$ are not singletons and:

\begin{equation}\label{e2}
\sum_{n\in \mathbb{N}}\sum_{a\in X_n} \frac{ \Big\vert\, 1-\sqrt{\mu_n(a) \big\vert\, X_n \,\big\vert} \,\Big\vert^2}{\big\vert\, X_n \,\big\vert} < \infty
\end{equation}
    
\end{theorem}

\begin{theorem}[{\cite[Lemma 2.14]{26}}, \cite{80}]\label{Theorem 1.28}

In the set-up of \textbf{Definition \ref{Definition 1.4}}, given an ITPFI factor $\mathcal{M} = R\big( \mathcal{M}_n, v_n \big)$ where for each $n\in\mathbb{N}$, $\mathcal{M}$ is an $I_{k_n}$ factor for some $2\leq k_n \leq \infty$, and:

$$
\operatorname{Sp}\big( v(n) \slash \mathcal{M}(n) \big) = \big\{ \lambda(n, i_n) \big\}_{1\leq i_n \leq k_n}
$$
Without losing generality, we can assume that for each $n\in \mathbb{N}$, $\lambda_{n, 1}$ is the maximum in $\operatorname{Sp}\big( v(n) \slash \mathcal{M}(n) \big)$. Then:

\begin{enumerate}[label = (\arabic*)]

    \item $\mathcal{M}$ is of type-$I$ if and only if the sequence $\big(1-\lambda_{n, 1} \big)$ is absolutely summable.
    \item $\mathcal{M}$ is of type-$II_1$ if and only if $k_n<\infty$ for all $n\in\mathbb{N}$ and:

    $$
    \sum_{n\in \mathbb{N}} \sum_{1\leq i_n \leq k_n} \frac{\vert\, 1-\sqrt{\lambda_{n, i}k_n} \,\vert^2}{k_n}
    $$
    \item $\mathcal{M}$ is of type-$III$ if and only if:

    $$
    \sum_{n\in \mathbb{N}} \sum_{1\leq i, j\leq k_n} \lambda_{n, i} \lambda_{n, j}\min\left( \left\vert\, \frac{\lambda_{n, i}}{\lambda_{n, j}} - 1\,\right\vert^2 \,,\, C \right)=\infty
    $$
    for some, and hence for all $C>0$.
    
\end{enumerate}
    
\end{theorem}

\section{Main theorems}

In this section, we will show that given a Bernoulli scheme, to find its ratio set, there is no need to consider all partial isomorphisms but only those that are of the form of bijections between two elementary cylinder sets of the same length. This is inspired by the definition of an asymptotic ratio set (\textbf{Definition \ref{Definition 1.6}}). We will start by discussing the correspondence between a Bernoulli scheme and an ITPFI factor.\\

\noindent
Given a Bernoulli scheme $G\curvearrowright (X, \mathcal{B}, \mu)$, for each $n\in\mathbb{N}$, suppose $2 \leq \vert\, X_n\,\vert = k_n \leq \infty$ and $\mu_n = \big( \mu_n(i) \big)_{1 \leq i \leq k_n}$ can be represented as a summable positive-valued sequence. Then for each $n\in\mathbb{N}$:

\begin{itemize}
    \item if $k_n<\infty$, define $\mathcal{M}_n = M_{k_n}(\mathbb{C})$ and let $\rho_n$ be the unique density matrix whose eigenvalues are $\big( \mu_n(i) \big)_{1\leq i \leq k_n}$. After changing coordinates, we can further assume that the set of eigenvalues of $\rho_n$ is $\big( \lambda_n(i) \big)_{1\leq i \leq k_n}$ where $\{\lambda_n(i): 1\leq i \leq k_n\} = \{ \mu_n(i):1 \leq i \leq k_n\}$ and $\lambda_n(1) \geq \lambda_n(2) \geq \cdots \lambda_n(k_n)$. Then let $v_n$ be the unit vector such that for all $S\in \mathcal{M}_n$, $\langle\,Sv_n \,,\, v_n\,\rangle = \operatorname{Tr}(S\rho_n)$. Put $\operatorname{Sp}\big( v_n \slash \mathcal{M}_n \big) = \big( \lambda_n(i) \big)_{1\leq i \leq k_n}$.

    \item if $k_n=\infty$, define $\mathcal{M}_n = B(\ell^2)$ and let $\rho_n$ be the diagonalizable operator whose eigenvalues are $\big( \mu_n(i) \big)_{i\in \mathbb{N}}$. By the previous reasoning, we can assume that $\big( \mu_n(i) \big)_{i\in \mathbb{N}}$ is decreasing. Since $\big( \mu_n(i) \big)_{i\in \mathbb{N}}$ is summable, $\rho_n$ is a trace-class operator. Then there exists a unit vector $v_n\in \ell^2$ such that for all $S\in \mathcal{K}(\ell^2)$, $\langle\,Sv_n \,,\, v_n\,\rangle = \operatorname{Tr}(S\rho_n)$.
    
\end{itemize}

\noindent
Conversely, given $\mathcal{M} = R\big( \mathcal{M}_n, v_n \big)$ an ITPFI factor, since for each $n\in\mathbb{N}$, each $v_n$ is a unit vector, its associated density matrix (or trace-class operator) $\rho_n$ satisfies $\operatorname{Tr}(\rho_n)=1$. Then for each $n\in\mathbb{N}$, $\operatorname{Sp}\big( v(n) \slash \mathcal{M}(n) \big)$ is a summable decreasing sequence indexed by $\{1, 2, \cdots, k_n\}$ where $2\leq k_n = \operatorname{dim}(\mathcal{M}_n) \leq \infty$, and hence defines a probability measure on $\{1, 2, \cdots, k_n\}$. For each $n\in\mathbb{N}$, by defining $X_n = \{1, 2, \cdots, k_n\}$ and $\mu_n(i) = \lambda_{n, i}$ for all $1\leq i \leq k_n$, the Bernoulli scheme on $(X, \mu)$ where $X=\prod_nX_n$ and $\mu = \otimes_{n\in \mathbb{N}} \mu_n$ is uniquely determined by $\mathcal{M}$. 

\begin{prop}[{\cite[Proposition 2.6]{73}}]\label{Proposition 2.1}

Given a Bernoulli scheme $G\curvearrowright (X, \mathcal{B}, \mu)$, let $\mathcal{M}$ be its associated ITPFI factor. If $\mu$ is of type-$III$, $r_{\infty}(\mathcal{M}) = r_G(\mu)$.
    
\end{prop}

\subsection{Ratio set of a type-\texorpdfstring{$III$}{III} Bernoulli scheme}

\begin{lem}\label{Lemma 2.1}

Given a type $III$ Bernoulli scheme $G\curvearrowright (X, \mathcal{B}, \mu)$, $\mu$ is of type-$III_0$ if and only if for any $\bm{x}\in X$ and $\epsilon\in (0, 1)$, there exists $\bm{y}\in X$ and there exists $K, N\in \mathbb{N}$ such that $N\geq M$, $x_n = y_n$ whenever $n\leq N$ and $n>N+K$, and;

$$
\min\Big( \big\vert\, D(\bm{x}, \bm{y}) - 1 \,\big\vert, \big\vert\, D(\bm{x}, \bm{y}) \,\big\vert \Big) < \epsilon
$$
where $D(\cdot \,,\, \cdot)$ is the Radon-Nikodym cocycle defined on the orbit equivalence relation.
    
\end{lem}

\begin{proof}

Suppose $\mu$ is of type-$III_0$. Fix $M\in \mathbb{N}$, $\epsilon\in (0, 1)$ and $\bm{x}\in C_M$ for some elementary cylinder set $C_M$ with length $M$. By definition, there exists $B(\epsilon, 1)$, $B(\epsilon, 2)\subseteq C_M$ and a partial isomorphism $\Phi: B(\epsilon, 1) \rightarrow B(\epsilon, 2)$ such that:

$$
\min\Big( \big\vert\, D(\bm{x}, \Phi(\bm{x}))-1\, \big\vert, \big\vert\, D(\bm{x}, \Phi(\bm{x})) \,\big\vert \Big) < \epsilon
$$
To prove the other direction, we will first show that the $0$ is in the ratio set (and clearly $1$ is in the ratio set). Fix $\bm{x}\in X$, $\epsilon\in \left(0, \dfrac{1}{2} \right)$ and $M\in \mathbb{N}$. By assumption, there exists $\bm{y}\in X$ and $N, K\in \mathbb{N}$ with $N\geq M$ such that $x_n=y_n$ whenever $n\leq N$ and $n>N+K$, and:

$$
\min\Big( \big\vert\, D(\bm{x}, \bm{y}) - 1\,\big\vert, \big\vert\, D(\bm{x}, \bm{y}) \,\big\vert \Big) < \epsilon
$$
Fix $C = C_{x_1,\,\cdots\, x_N}$ and define:

$$
C_1 = C_{x_1,\,\cdots\, x_N, x_{N+1},\,\cdots\, x_{N+K}}, \hspace{1cm} C_2 = C_{x_1,\,\cdots\, x_N,\, y_{N+1},\, \cdots\, y_{N+K}}
$$
Then define the following mapping:

$$
\varphi: C_1 \rightarrow C_2, \hspace{0.3cm} \bm{x}\mapsto
\begin{cases}
\big( \varphi(\bm{x}) \big)_n = x_n, \hspace{1cm} n\notin\big\{ N+1,\,\cdots\, N+K \big\}\\
\big( \varphi(\bm{x}) \big)_n = y_n, \hspace{1cm} n\in \big\{ N+1,\, \cdots\, N+K \big\}
\end{cases}
$$
Clearly $\varphi$ is a partial isomorphism and for $\mu$-almost each $\bm{s}\in C_2$, we have either $\left\vert\, \dfrac{d\,\varphi\mu}{d\,\mu} (\bm{s})-1 \,\right\vert < \epsilon$ or $\left\vert\, \dfrac{d\,\varphi\mu}{d\,\mu} (\bm{s}) \,\right\vert < \epsilon$. Hence:

$$
\min\left( \left\vert\, \dfrac{d\,\varphi\mu}{d\,\mu} (\bm{x})-1 \,\right\vert, \left\vert\, \dfrac{d\,\varphi\mu}{d\,\mu} (\bm{x}) \,\right\vert \right) = \min\Big( \big\vert\, D(\bm{x}, \bm{y}) - 1\,\big\vert, \big\vert\, D(\bm{x}, \bm{y}) \,\big\vert \Big) < \epsilon
$$
Next fix an arbitrary measurable subset $B$ with $\mu(B)>0$. Fix $\bm{x}\in B$, $M\in \mathbb{N}$, and let $\bm{y}\in X$, $N, K\in\mathbb{N}$ be given by the assumption. Since $\mu(B)>0$ and $\bm{x}\in B$, according to {\cite[\textbf{Martingale Convergence Theorem} ]{1}} there exists $N$ large enough (also greater than $M$) such that:

$$
\mu\big( B\cap C_{x_1,\,\cdots\, x_N} \big)>0
$$
Therefore:

$$
\begin{aligned}
& \mu\big( B\cap C_{s_1,\,\cdots\, s_N,\, x_{N+1},\,\cdots\, x_{N+K}} \big) = \mu\big( B\cap C_{s_1,\,\cdots\, s_N} \big) \prod_{1\leq i \leq K} \mu_{N+i}(x_{N+i}) > 0\\
& \mu\big( B\cap C_{s_1,\,\cdots\, s_N,\, y_{N+1},\,\cdots\, y_{N+K}} \big) = \mu\big( B\cap C_{s_1,\,\cdots\, s_N} \big) \prod_{1\leq i \leq K} \mu_{N+i}(y_{N+i}) > 0\\
\end{aligned}
$$
Next define:

$$
B_1 = B\cap C_{s_1,\,\cdots\, s_N,\, x_{N+1},\,\cdots\, x_{N+K}}, \hspace{1cm}
B_2 = B\cap C_{s_1,\,\cdots\, s_N,\, y_{N+1},\,\cdots\, y_{N+K}}
$$
and define $\psi: B_1\rightarrow B_2$ the same way we define $\varphi$. We then can conclude that for $\mu$-almost every $\bm{s}\in B_2$, we have either $\left\vert\, \dfrac{d\,\psi\mu}{d\,\mu}(\bm{s}) \,\right\vert < \epsilon$ or $\left\vert\, \dfrac{d\,\psi\mu}{d\,\mu}(\bm{s})-1 \,\right\vert < \epsilon$. If $\psi(\bm{x}) = \bm{y}$, then we are good. Otherwise:

\begin{itemize}

    \item if $\bm{y}\in B_2$ but $\bm{y}\neq \psi(\bm{x})$, then define the following mapping:

    $$
    \psi': B_1 \rightarrow B_2, \hspace{0.3cm} \bm{s}\mapsto 
    \begin{cases}
    \psi(\bm{s}), \hspace{1.08cm} \bm{s}\notin \big\{ \bm{x}\,,\, \psi^{-1}(\bm{y}) \big\}\\
    \bm{y}, \hspace{1.54cm} \bm{s} = \bm{x}\\
    \psi(\bm{x}), \hspace{1cm} \bm{s} = \psi^{-1}(\bm{y})
    \end{cases}
    $$
    Clearly $\psi'$ is a partial isomorphism. Observe that:
    
    $$
    \mu\Big( B(\epsilon, 1)\backslash \big\{ \bm{x}, \psi^{-1}(\bm{y}) \big\} \Big) = \mu\big( B(\epsilon, 1) \big)
    $$
    Since $\psi'$ coincides with $\psi$ on $B(\epsilon, 1)\backslash \big\{ \bm{x}, \psi^{-1}(\bm{y}) \big\}$, we then can conclude for $\mu$-almost each $\bm{s}\in B(\epsilon, 2)$, $\left\vert\, \dfrac{d\,\psi'\mu}{d\,\mu}(\bm{s}) \,\right\vert < \dfrac{\epsilon}{2}$ or $\left\vert\, \dfrac{d\,\psi'\mu}{d\,\mu}(\bm{s}) - 1 \,\right\vert < \dfrac{\epsilon}{2}$

    \item if $\bm{y} \notin B_2$, define the following mapping:

    $$
    \psi'': B_1 \rightarrow \big( B_2 \cup\{y\} \big) \backslash \big\{ \psi(\bm{x}) \big\} \hspace{0.3cm} \bm{s}\mapsto 
    \begin{cases}
    \psi(\bm{s}), \hspace{1cm} \bm{s}\neq \bm{x}\\
    \bm{y}, \hspace{1.46cm} \bm{s} = \bm{x}
    \end{cases}
    $$
    Clearly $\psi''$ is a partial isomorphism. Since $\mu\big( B_1 \big) = \mu\big( B_1 \backslash \{\bm{x}\} \big)$ and $\psi$ coincides with $\psi''$ on $B_1 \backslash \{x\}$, we then have $\left\vert\, \dfrac{d\,\psi''\mu}{d\,\mu}(\bm{s}) \,\right\vert < \dfrac{\epsilon}{2}$ or $\left\vert\, \dfrac{d\,\psi''\mu}{d\,\mu}(\bm{s}) - 1 \,\right\vert < \dfrac{\epsilon}{2}$ for $\mu$-almost each $\bm{s}\in B_2$.
    
\end{itemize}

\noindent
In general, by replacing $\psi$ with another partial isomorphism that coincides with $\psi$ for $\mu$-almost every point in $B_1$, we can assume that $\psi(\bm{x}) = \bm{y}$, which implies that either $\big\vert\, D(\bm{x}, \bm{y}) - 1 \,\big\vert < \epsilon$ or $\big\vert\, D(\bm{x}, \bm{y}) \,\big\vert < \epsilon$. We can now show that both $0$ and $1$ are in the ratio set. Next we need to show that there are no other values in the ratio set. Assume by contradiction that $r\in(0, 1)$ is in the ratio set. Fix an $\epsilon\in (0, 1)$ such that $\epsilon< \min(r, 1-r)$. Fix $\bm{x}\in X$, $M\in \mathbb{N}$ and an elementary cylinder set $C_M$ with $\bm{x}\in C_M$. Then by definition there exists two non-negligible sets $B(\epsilon, 1)$, $B(\epsilon, 2) \subseteq C_M$, and $\Psi: B(\epsilon, 1) \rightarrow B(\epsilon, 2)$ a partial isomorphism such that $\left\vert\, \dfrac{d\,\Psi\mu}{d\,\mu} (\bm{s}) - r\,\right\vert < \epsilon$ for $\mu$-almost every $\bm{s}\in B(\epsilon, 2)$. Let $\bm{y}\in X$ and $N, K\in \mathbb{N}$ be given by the assumption. By the same reasoning above, we can assume that $\Psi(\bm{x}) = \bm{y}$ so that:

$$
\min\left( \left\vert\, \frac{d\,\Psi\mu}{d\,\mu}(\bm{y}) - 1 \,\right\vert, \left\vert\, \frac{d\,\Psi\mu}{d\,\mu}(\bm{y}) \,\right\vert \right) = \min\Big( \big\vert\, D(\bm{x}, \bm{y}) - 1\,\big\vert, \big\vert\, D(\bm{x}, \bm{y}) \,\big\vert \Big) < \epsilon
$$
which is absurd since $\epsilon<\min(r, 1-r)$. Therefore, we can now conclude that the ratio set is equal to $\{0, 1\}$.
    
\end{proof}

\begin{lem}\label{Lemma 2.2}

Given a type-$III$ Bernoulli scheme $G\curvearrowright (X, \mathcal{B}, \mu)$, $\mu$ is of type-$III_1$ if and only for any $r\in (0, 1)$, $M\in \mathbb{N}$, $\bm{x}\in X$, and $\epsilon\in (0, r)$, there exists $\bm{y}\in X$ and there exists $K, N\in \mathbb{N}$ with $N\geq M$ such that $x_n = y_n$ whenever $n\leq N$ and $n>N+K$, and such that:

$$
\big\vert\ D(\bm{x}, \bm{y}) - r \,\big\vert < \epsilon
$$
where $D(\cdot \,,\, \cdot)$ is the Radon-Nikodym cocycle defined on the orbit equivalence.
    
\end{lem}

\begin{proof}

If $\mu$ is of type-$III_1$, given $r\in (0, 1)$, $M\in \mathbb{N}$, $\bm{x}\in X$ and $\epsilon\in (0, r)$, let $C_M$ be an elementary cylinder set with length $M$ such that $\bm{x}\in C_M$. Then since $r$ is in the ratio set, there exists two non-negligible sets $B_1, B_2\subseteq C_M$, and a partial isomorphism $\phi: B_1\rightarrow B_2$ such that $\big\vert\, D\big( \bm{x}, \phi(\bm{x}) \big) -r\,\big\vert < \epsilon$. Then $\phi(\bm{x})$ is the desired $\bm{y}$. Conversely, fix an arbitrary $r\in (0, 1)$, $\epsilon\in (0, r)$, $M\in \mathbb{N}$ and $\bm{x}\in X$. Then fix a measurable subset $B\subseteq X$ with $\mu(B)>0$. By the same technique used in the proof of \textbf{Lemma \ref{Lemma 2.1}}, there exists two non-negligible subsets $B_1$, $B_2\subseteq B$ with $\bm{x}\in B_1$, $\bm{y}\in B_2$ and a partial isomorphism $\Phi$ (defined the same way as $\varphi$ in the proof of \textbf{Lemma \ref{Lemma 2.1}}) such that $\Phi(\bm{x}) = \bm{y}$ and $\left\vert\, \dfrac{d\,\Phi\mu}{d\,\mu} (\bm{s}) - r \,\right\vert < \epsilon$ for $\mu$-almost every $\bm{s}\in B_2$. We can then conclude that $r\in (0, 1)$ and, since $r$ is arbitrarily picked, that system is of type-$III_1$.
    
\end{proof}

\subsection{An alternate way to describe a type-\texorpdfstring{$III$}{III} Bernoulli scheme}

Throughout this section, we will use $G\curvearrowright (X, \mathcal{B}, \mu)$ to denote a type-$III$ Bernoulli scheme. The ideas of the proof of the case where $\limsup_n \vert\, X_n\,\vert = \infty$ are inspired by {\cite[\textbf{Proposition 2.1}]{7}}. We will first assume that each $X_n$ is well-ordered, namely $X_n = \mathbb{N} \cup \{0\}$ whenever $X_n$ is infinite and $X_n = \mathbb{Z}_{m_n}$ whenever $\vert\ X_n\,\vert = m_n$ for some $m_n\in \mathbb{N}$. We will further assume that for each $n\in\mathbb{N}$, $\mu_n(0)$ is the maximum weight of $\mu$. Then we will show that our main theorems are independent to permutation of each $X_n$, and hence can be applied to an arbitrary Beroulli scheme.

\subsubsection{When $\limsup_n\vert\, X_n\,\vert = \infty$}
$\hspace{1cm}$\\

Observe that $\vert\, X_n\,\vert$ is greater than $2$ if and only if the set $\{0,1,2\} \subseteq X_n$. Fix $N\in \mathbb{N}$ with $\vert\, X_N \,\vert > 2$. For any $m\in \mathbb{N}$ with $\vert\, X_m\,\vert > 2$ and $m\neq N$, define:

\begin{equation}\label{e43}
\sigma_{N, m}: X\longrightarrow X, \hspace{0.2cm} \bm{x} \mapsto
\begin{cases}
\big( \sigma_{N, m}(\bm{x}) \big)_N = 2\, \text{  and  }\, \big( \sigma_{N, m}(\bm{x}) \big)_m=0, \hspace{1cm} \text{  if  } x_N=0\, \text{  and  }\, x_m=1\\
\big( \sigma_{N, m}(\bm{x}) \big)_N = 0\, \text{  and  }\, \big( \sigma_{N, m}(\bm{x}) \big)_m=1, \hspace{1cm} \text{  if  } x_N=2\, \text{  and  }\, x_m=0\\
\big( \sigma_{N, m}(\bm{x}) \big)_n=x_n, \hspace{4.42cm} \text{otherwise}
\end{cases}
\end{equation}
Recall that in \textbf{Definition \ref{Definition 1.20}}, $L^{\infty} \big( X\times \mathbb{R}, d\mu\times d\nu \big)^G$ denotes all $G$-invariant essentially bounded functions defined on the Maharam extension of $G\curvearrowright (X, \mathcal{B}, \mu)$ where $d\nu(t) = \operatorname{exp}(-\vert\, t\,\vert)dt$. For any $F\in L^{\infty} \big( X\times \mathbb{R}, d\mu\times d\nu \big)^G$ and an arbitrary partial isomorphism $\Phi$, by the definition of partial isomorphism we have for any $\bm{x}$ in the domain of $\Phi$ and any $t\in \mathbb{R}$:

$$
F\Big( \Phi(\bm{x}), t+\log D\big( \bm{x}, \Phi(\bm{x}) \big) \Big) = F(\bm{x}, t)
$$
In particular, for any two different $N, m\in \mathbb{N}$ with $I_N, I_m>2$, we have that for all $\bm{x}\in X$ and $t\in \mathbb{R}$:

\begin{equation}\label{e4}
F\big( \sigma_{N, m}(\bm{x}), t \big) = F\Big( \bm{x}, t\pm\log D\big( \bm{x}, \sigma_{N, m}(\bm{x}) \big) \Big) = F\left( \bm{x}, t\pm\log \frac{\mu_N(2) \mu_m(0)}{\mu_N(0) \mu_m(1)} \right)
\end{equation}
Before proving the key lemma, first we introduce the following notations. For each $N\in \mathbb{N}$, define:

$$
X_N^* = \prod_{n\in \mathbb{N} \backslash \{N\}} X_n, \hspace{1cm} \mu_N^* = \bigotimes_{n\in \mathbb{N} \backslash \{N\}} \mu_n
$$
Then $(X, \mu)$ can be identified by $(X_N, \mu_N) \times \big( X_N^*, \mu_N^* \big)$ and each $\bm{x}\in X$ can be rewritten as $\bm{x} = (x_N, x_N^*)$. We let $\nu$ be the finite measure defined on $\mathbb{R}$ such that $d\nu = e^{-\vert\, t\,\vert}dt$.

\begin{lem}\label{Lemma 2.4}

For any $F\in L^{\infty}\big( X\times \mathbb{R}, \mu\times\nu \big)^G$, $\epsilon\in (0, 1)$ and $m\in\mathbb{N}$ with $I_m>2$, for all large enough $N\in\mathbb{N}$ we have:

\begin{equation}\label{e42}
\left\vert\, F\big( (2, \bm{x}_N^*), t \big) - F\left( (0, \bm{x}_N^*), t+\log\frac{\mu_m(1)}{\mu_m(0)} - \log\frac{\mu_N(2)}{\mu_N(0)} \right) \,\right\vert < 2\epsilon
\end{equation}
for $\mu_N^*$-almost every $\bm{x}_N^*\in X_N^*$ and $\nu$-almost every $t\in\mathbb{R}$.
    
\end{lem}

\begin{proof}

For any $N\in \mathbb{N}$ we identify $(X, \mu)$ as $\big( X_N, \mu_N \big) \times \big( X_N^*, \mu_N^* \big)$ and for each $\bm{x}\in X$, we rewrite $\bm{x} = (x_N, \bm{x}_N^*)$. Fix $m\in \mathbb{N}$. For any finite subset $\mathcal{F} \subseteq \mathbb{N}$ with $m\in \mathcal{F}$, $N\in \mathbb{N}$ with $N>\max \mathcal{F}$, any $H\in L^{\infty} \big( X_{\mathcal{F}} \times \mathbb{R}, \mu_{\mathcal{F}} \times\nu \big)$ and $F_0 \in L^{\infty}\big( X_N\times X_{\mathcal{F}} \times \mathbb{R}, \mu_N \times \mu_{\mathcal{F}} \times \nu \big)$, observe that:

\begin{equation}\label{e5}
\begin{aligned}
& \hspace{0.45cm} \int_{X_{\mathcal{F}} \times \mathbb{R}} F_0\big( \sigma_{N, m} (0, x), t \big) H(x, t) d\mu_{\mathcal{F}}(x)\, d\nu(t)\\
& = \int_{X_{\mathcal{F}} \times \mathbb{R}} \Big( \big( 1-\mu_m(1) \big) F_0(0, x, t) + \mu_m(1) F_0(2, x, t) \Big) H(x, t) d\mu_{\mathcal{F}}(x)\, d\nu(t)
\end{aligned}
\end{equation}
where $\sigma_{N, m}$ is defined in (\ref{e43}). Since $X_{\mathcal{F}}$ we pick in (\ref{e5}) is a finite union of cylinder sets, both $H$ and $F_0$ in (\ref{e5}) can be viewed as simple functions defined on $X_N^*$. Also $\mu_N^*$ restricted to $X_{\mathcal{F}}$ is equal to $\mu_{\mathcal{F}}$. Then (\ref{e5}) can be rewritten as: 

\begin{equation}\label{e6}
\begin{aligned}
& \hspace{0.45cm} \int_{X_N^*\times \mathbb{R}} F_0\big( \sigma_{N, m} (0, \bm{x}_N^*), t \big) H(\bm{x}_N^*, t) d\mu_N^* (\bm{x}_N^*)\, d\nu(t)\\
& = \int_{X_N^* \times \mathbb{R}} \Big( \big( 1-\mu_m(1) \big) F_0\big( (0, \bm{x}_N^*), t \big) + \mu_m(1) F_0\big( (2, \bm{x}_N^*), t \big) \Big) H(\bm{x}_N^*, t) d\mu_N^* (\bm{x}_N^*)\, d\nu(t)
\end{aligned}
\end{equation}
Since $\mu\times \nu$ is finite and the family of cylinder sets generated the $\sigma$-algebra of $X$, $F$ can be arbitrarily approximated in $L^1$-norm by a function $F_0$ defined on $X_N \times X_{\mathcal{F}} \times \mathbb{R}$ for some finite set $\mathcal{F}$. Fix $\epsilon\in (0, 1)$. Then there exists $\mathcal{F} \subseteq \mathbb{N}$ a finite set such that $m\in\mathcal{F}$ and:

\begin{equation}\label{e7}
\begin{aligned}
& \int_{X_N^* \times\mathbb{R}} \Big\vert\, F\big( (2,\bm{x}_N^*), t \big) - F_0\big( (2, \bm{x}_N^*), t \big) \,\Big\vert H(\bm{x}_N^*, t) d\mu_N^*(\bm{x}_N^*) d\nu(t) < \epsilon\\
& \int_{X_N^* \times\mathbb{R}} \Big\vert\, F\big( (0,\bm{x}_N^*), t \big) - F_0\big( (0, \bm{x}_N^*), t \big) \,\Big\vert H(\bm{x}_N^*, t) d\mu_N^*(\bm{x}_N^*) d\nu(t) < \epsilon\\
\end{aligned}
\end{equation}
where $H\in L^{\infty}\big( X_{\mathcal{F}} \times\mathbb{R}, \mu_{\mathcal{F}} \times\nu \big)$ is arbitrarily picked and $N\in\mathbb{N}$ is large enough so that $N\notin \mathcal{F}$. By (\ref{e4}), we also have:

\begin{equation}\label{e8}
\begin{aligned}
& \hspace{0.45cm} \int_{X_N^* \times \mathbb{R}} F\big( \sigma_{N, m} (0, \bm{x}_N^*), t \big) H(\bm{x}_N^*, t) \,d\mu(\bm{x}_N^*) \,d\nu(t)\\
& = \int_{X_N^* \times \mathbb{R}} F\Big( (0, \bm{x}_N^*), t- \log D\big( \sigma_{N, m} (0, \bm{x}_N^*), (0, \bm{x}_N^*) \big) \Big)H(\bm{x}_N^*, t) \,d\mu(\bm{x}_N^*) d\nu(t)\\
& = \int_{X_N^*\times \mathbb{R}} \big( 1-\mu_m(1) \big) F\big( (0, \bm{x}_N^*), t \big) + \mu_m(1) F\left( (0, \bm{x}_N^*), t-\log \frac{\mu_N(2)\mu_m(0)}{\mu_N(0) \mu_m(1)} \right)H(\bm{x}_N^*, t)  \,d\mu(\bm{x}_N^*) d\nu(t)\\
& = \int_{X_N^*\times \mathbb{R}} \big( 1-\mu_m(1) \big) F\big( (0, \bm{x}_N^*), t \big) + \mu_m(1) F\left( (0, \bm{x}_N^*), t+ \log \frac{\mu_m(1)}{\mu_m(0)} - \log \frac{\mu_N(2)}{\mu_N(0)} \right) H(\bm{x}_N^*, t) \,d\mu(\bm{x}_N^*) d\nu(t)
\end{aligned}
\end{equation}
Combining (\ref{e6}), (\ref{e7}) and (\ref{e8}), we have:

\begin{equation}\label{e9}
\int_{X_N^* \times\mathbb{R}} \big( 1 - \mu_m(1) \big) \left\vert\, F\big( (2, \bm{x}_N^*), t \big) - F\left( (0, \bm{x}_N^*), t+ \log \frac{\mu_m(1)}{\mu_m(0)} - \log \frac{\mu_N(2)}{\mu_N(0)} \right) \,\right\vert H(\bm{x}_N^*, t) \,d\mu_N^*(\bm{x}_N^*) d\nu(t) < 2\epsilon
\end{equation}
Since (\ref{e9}) holds for all $H(\bm{x}_N^*, t)\in L^{\infty}\big( X_N^*\times \mathbb{R}, \mu_N^*\times \nu \big)$, we can now conclude for $\mu_N^*$-almost every $\bm{x}_N^* \in X_N^*$ and $\nu$-almost every $t\in \mathbb{R}$:

\begin{equation}\label{e10}
\left\vert\, F\big( (2, \bm{x}_N^*), t \big) - F\left( (0, \bm{x}_N^*), t+ \log \frac{\mu_m(1)}{\mu_m(0)} - \log \frac{\mu_N(2)}{\mu_N(0)} \right) \,\right\vert < 2\epsilon
\end{equation}

\end{proof}

\begin{defn}\label{Definition 2.5}

Given a Bernoulli scheme $G\curvearrowright (X, \mathcal{B}, \mu)$ where $X=\prod_{n\in \mathbb{N}} X_n$ and each $X_n = \big\{0, 1, \cdots, \vert\, X_n\,\vert-1 \big\}$, for each $i\in \mathbb{N} \cup \{0\}$, define:

$$
\mathcal{I}_i = \big\{ n\in \mathbb{N} \cup \{0\}: i\in X_n \big\}, \hspace{1cm} \mathcal{N} = \big\{ i\in \mathbb{N} \cup \{0\}: \vert\, \mathcal{I}_i \,\vert = \infty \big\}
$$
By our assumption in this section, we will only consider the Bernoulli scheme where $\mathcal{I}_0 = \mathcal{I}_1 = \mathbb{N}$ and $\vert\, \mathcal{N} \,\vert = \infty$. For generality, we will assume that $\mathbb{N} \backslash \mathcal{N}$ is also infinite in the rest of the section.
    
\end{defn}

\begin{lem}\label{Lemma 2.6}

For any $F\in L^{\infty} \big( X\times \mathbb{R}, \mu\times\nu \big)^G$, $\epsilon\in (0, 1)$, $i\in \mathcal{N}$ and $m, k\in \mathbb{N}$ with $I_m, I_k>2$, for all $N\in\mathcal{I}_i$ that is large enough, we have:

$$
\left\vert\, F\big( (i, \bm{x}_N^*), t \big) - F\left( (i, \bm{x}_N^*), t+\log\frac{\mu_m(1)}{\mu_m(0)} - \log\frac{\mu_k (2)}{\mu_k (0)} \right) \,\right\vert < \epsilon
$$
for $\mu_N^*$-almost every $\bm{x}_N^*\in X_N^*$ and $\nu$-almost every $t\in \mathbb{R}$.
    
\end{lem}

\begin{proof}

Fix $F\in L^{\infty}\big( X\times \mathbb{R}, \mu\times\nu \big)^G$, $\epsilon\in (0, 1)$. Fix $m\in \mathbb{N}$ with $I_m>2$. According to \textbf{Lemma \ref{Lemma 2.4}}, there exists $N\in\mathbb{N}$ with $I_N>2$ such that:

$$
\left\vert\, F\big( (2, \bm{x}_N^*), t \big) - F\left( (0, \bm{x}_N^*), t+\log\frac{\mu_m(1)}{\mu_m(0)} - \log\frac{\mu_N(2)}{\mu_N(0)} \right) \,\right\vert < \frac{\epsilon}{2}
$$
for $\mu_N^*$-almost every $\bm{x}_N^*\in X_N^*$ and $\nu$-almost every $t\in \mathbb{R}$. Next fix $k\in \mathbb{N}$ with $I_k>2$. According to the proof of \textbf{Lemma \ref{Lemma 2.4}}, the $N$ we pick for (\ref{e10}) to hold can be arbitrarily large. Then for any $N\in\mathcal{I}_i$ with $N>k$, define: 

$$
\tau_{N, k}: X\longrightarrow X, \hspace{0.2cm} \bm{x} \mapsto
\begin{cases}
\big( \tau_{N, k}(\bm{x}) \big)_N = 2\, \text{  and  }\, \big( \tau_{N, k}(\bm{x}) \big)_k=0, \hspace{1cm} \text{  if  } x_N=0\, \text{  and  }\, x_k=2\\
\big( \tau_{N, k}(\bm{x}) \big)_N = 0\, \text{  and  }\, \big( \tau_{N, k}(\bm{x}) \big)_k=2, \hspace{1cm} \text{  if  } x_N=2\, \text{  and  }\, x_k=0\\
\big( \tau_{N, k}(\bm{x}) \big)_n=x_n, \hspace{4.2cm} \text{otherwise}
\end{cases}
$$
Similar to $\sigma_{N, m}$, $\tau_{N, k}$ is a partial isomorphism. Observe that for $\mu$-almost each $\bm{x}\in X$:

$$
\log D\big( \bm{x}, \tau_{N, k}(\bm{x}) \big) = \log\frac{\mu_N(2)\mu_k(0)}{\mu_N(0)\mu_k(2)}
$$
Hence, replacing $\sigma_{N, m}$ by $\tau_{N, k}$ in the proof of \textbf{Lemma \ref{Lemma 2.4}}, we can then assume $N$ is large enough so that:

\begin{equation}\label{e11}
\left\vert\, F\big( (2, \bm{x}_N^*), t \big) - F\left( (0, \bm{x}_N^*), t+\log\frac{\mu_k(2)}{\mu_k(0)} - \log\frac{\mu_N(2)}{\mu_N(0)} \right) \,\right\vert < \frac{\epsilon}{2}
\end{equation}
Combining (\ref{e10}) and (\ref{e11}), we then have for $\mu_N^*$-almost every $\bm{x}_N^*$ and $\nu$-almost every $t\in \mathbb{R}$:

$$
\left\vert\, F\left( (0, \bm{x}_N^*), t+\log\frac{\mu_k(2)}{\mu_k(0)} - \log\frac{\mu_N(2)}{\mu_N(0)} \right) - F\left( (0, \bm{x}_N^*), t+\log\frac{\mu_m(1)}{\mu_m(0)} - \log\frac{\mu_N(2)}{\mu_N(0)} \right) \,\right\vert < \epsilon
$$
which implies for $\mu_N^*$-almost every $\bm{x}_N^*\in X_N^*$ and $\nu$-almost every $t\in \mathbb{R}$:

$$
\left\vert\, F\big( (0, \bm{x}_N^*), t \big) - F\left( (0, \bm{x}_N^*), t+\log\frac{\mu_m(1)}{\mu_m(0)} - \log\frac{\mu_k(2)}{\mu_k(0)} \right) \,\right\vert < \epsilon
$$
By assumption, $\{0, 1, 2\}\subseteq\mathcal{N}$, so the desired inequality for an arbitrary $i\in \mathcal{N}$ can be proved by the same reasoning. Fix $i\in \mathcal{N}$ and for any $N\in \mathcal{I}_i$ with $N>m$, $N>k$, define:

$$
\sigma_{N, m}^i: X\longrightarrow X, \hspace{0.2cm} \bm{x} \mapsto
\begin{cases}
\big( \sigma^i_{N, m}(\bm{x}) \big)_N = 2\, \text{  and  }\, \big( \sigma^i_{N, m}(\bm{x}) \big)_m=0, \hspace{1cm} \text{  if  } x_N=i\, \text{  and  }\, x_m=1\\
\big( \sigma^i_{N, m}(\bm{x}) \big)_N = i\, \text{  and  }\, \big( \sigma^i_{N, m}(\bm{x}) \big)_m=1, \hspace{1.05cm} \text{  if  } x_N=2\, \text{  and  }\, x_m=0\\
\big( \sigma^i_{N, m}(\bm{x}) \big)_n=x_n, \hspace{4.42cm} \text{otherwise}
\end{cases}
$$
and:

$$
\tau_{N, k}^i: X\longrightarrow X, \hspace{0.2cm} \bm{x} \mapsto
\begin{cases}
\big( \tau^i_{N, k}(\bm{x}) \big)_N = 2\, \text{  and  }\, \big( \tau^i_{N, k}(\bm{x}) \big)_k=0, \hspace{1cm} \text{  if  } x_N=i\, \text{  and  }\, x_k=2\\
\big( \tau^i_{N, k}(\bm{x}) \big)_N = i\, \text{  and  }\, \big( \tau^i_{N, k}(\bm{x}) \big)_k=2, \hspace{1.06cm} \text{  if  } x_N=2\, \text{  and  }\, x_k=0\\
\big( \tau^i_{N, k}(\bm{x}) \big)_n=x_n, \hspace{4.2cm} \text{otherwise}
\end{cases}
$$
Observe that for $\mu$-almost every $\bm{x}\in X$:

$$
\log D\big( \bm{x}, \sigma_{N, m}^i(\bm{x}) \big) = \log \frac{\mu_N(2)\mu_m(0)}{\mu_N(i)\mu_m(1)}, \hspace{1cm}
\log D\big( \bm{x}, \tau_{N, k}^i(\bm{x}) \big) = \log\frac{\mu_N(2) \mu_k(0)}{\mu_N(i) \mu_k(2)}
$$
Similarly, replacing $\sigma_{N, m}$ by $\sigma_{N, m}^i$ or $\tau_{N, k}^i$ in the proof of \textbf{Lemma \ref{Lemma 2.4}}, there exists $N\in \mathcal{I}_i$ with:

$$
\begin{aligned}
& \hspace{1cm} \left\vert\, F\big( (2, \bm{x}_N^*), t \big) - F\left( (i, \bm{x}_N^*), t+\log\frac{\mu_m(1)}{\mu_m(0)} - \log\frac{\mu_N(2)}{\mu_N(i)} \right) \,\right\vert < \frac{\epsilon}{2} \\
& \hspace{1cm} \left\vert\, F\big( (2, \bm{x}_N^*), t \big) -F\left( (i, \bm{x}_N^*), t+\log\frac{\mu_k(2)}{\mu_k(0)} - \log\frac{\mu_N(2)}{\mu_N(i)} \right) \,\right\vert < \frac{\epsilon}{2}
\end{aligned}
$$
which implies $\mu_N^*$-almost every $\bm{x}_N^*\in X_N^*$ and $\nu$-almost every $t\in \mathbb{R}$:

\begin{equation}\label{e14}
\left\vert\, F\big( (i, \bm{x}_N^*), t \big) - F\left( (i, \bm{x}_N^*), t+\log\frac{\mu_m(1)}{\mu_m(0)} - \log\frac{\mu_k(2)}{\mu_k(0)} \right) \,\right\vert < \epsilon
\end{equation}
    
\end{proof}

\begin{rem}\label{Remark 2.7}

In \textbf{Lemma \ref{Lemma 2.6}}, by adjusting the definition of the partial isomorphism $\sigma_{N, m}^i$ and $\tau_{N, k}^i$, we further have that for any $\epsilon\in (0, 1)$, for any $i, j\in\mathcal{N}$ with $i, j \neq 0$ and for any $m\in \mathcal{I}_i$, $k\in \mathcal{I}_j$, we have whenever $N\in \mathbb{N}$ is large enough:

\begin{equation}\label{e12}
\left\vert\, F\big( (l, \bm{x}_N^*), t\big) - F\left( \big(l, \bm{x}_N^*), t+\log \frac{\mu_m(i)}{\mu_m(0)} + \log \frac{\mu_k(j)}{\mu_k(0)} \right) \,\right\vert < \epsilon
\end{equation}
for any $l\in X_N$, $\nu_N^*$-almost every $\bm{x}_N^*\in X_N^*$ and for $\nu$-almost every $t\in \mathbb{R}$ whenever $N\in\mathbb{N}$ is large enough. Suppose for some $i, j\in \mathcal{N}$ and for some $r, s\in(0, 1)$, there exists $(m_n)_{n\in\mathbb{N}}$, $(k_n)_{n\in \mathbb{N}}$ two infinite subsets of $\mathbb{N}$ such that $(m_n)_{n\in \mathbb{N}} \subseteq \mathcal{I}_i$, $(k_n)_{n\in \mathbb{N}} \subseteq \mathcal{I}_j$ and:

$$
\lim_n \frac{\mu_{m_n}(i)}{\mu_{m_n}(0)} = r, \hspace{1cm} \lim_n \frac{\mu_{k_n}(j)}{\mu_{k_n}(0)} = s
$$
Let $(\epsilon_n)_{n\in \mathbb{N}} \subseteq (0, 1)$ with $\lim_n \epsilon_n=0$. Without losing generality (or restrict to a subsequence), suppose for all $n\in\mathbb{N}$:

$$
\left\vert\, \log\frac{\mu_{m_n}(i)}{\mu_{m_n}(0)} - \log r \,\right\vert =\left\vert\, \log\frac{\mu_{k_n}(j)}{\mu_{k_n}(0)} - \log r \,\right\vert = o\left( \frac{1}{n} \right)
$$
Therefore:

\begin{equation}\label{e13}
\begin{aligned}
& \hspace{0.4cm} \left\vert\, F\big( (l, \bm{x}_N^*), t \big) - F\left( (l, \bm{x}_N^*), t+\log \frac{\mu_{m_n}(i)}{\mu_{m_n}(0)} \right) + \log \frac{\mu_{k_n}(j)}{\mu_{k_n}(0)} \,\right\vert \\
& = \left\vert\, F\big( (l, \bm{x}_N^*), t \big) - F\left( (l, \bm{x}_N^*), t+\log r +\log s \pm o\left(\frac{2}{n} \right) \right) \,\right\vert < \epsilon_n
\end{aligned}
\end{equation}
for each $l\in X_N$, $\nu$-almost every $t\in \mathbb{R}$ and $\bm{\mu}_N^*$-almost every $\bm{x}_N^*\in X_N^*$ whenever $N\in \mathbb{N}$ is large enough. According to (\ref{e14}), whenever (\ref{e13}) holds, we can replace $(0, \bm{x}_N^*)$ in (\ref{e13}) by $(l, \bm{x}_N^*)$ for any $l\in X_N$.
    
\end{rem}

\begin{rem}\label{Remark 2.8}

To prove the main theorem, we will need the following result:
given an integrable function $f\in L^1\big( \mathbb{R}, \nu\big)$, there exists a sequence of differentiable function $(f_n)_{n\in\mathbb{N}}$ such that $\lim_n \|f_n-f\|_1 = 0$. Then there exists a strictly increasing sequence of integers $(N_k)_{k\in \mathbb{N}}$ such that for each $k\in\mathbb{N}$, $\|f_m-f_n\|_1 < \dfrac{1}{2^k}$ whenever $m, n\geq N_k$. Observe that for each $k\in\mathbb{N}$:

$$
f_{N_k} = f_{N_1} + \sum_{1\leq i < k} f_{N_{i+1}} - f_{N_i}, \hspace{1cm} g_k = \big\vert\, f_{N_1} \,\big\vert + \sum_{1\leq i \leq k} \big\vert\, f_{N_{i+1}} - f_{N_i} \,\big\vert
$$
Clearly $\sup_k \|g_k\|_1 < \infty$. Hence $f_{N_k}$ converges absolutely (hence unconditionally) $\nu$-almost everywhere to some $\tilde{f}\in L^1\big( \mathbb{R}, \nu \big)$. By \textbf{Bounded Convergence Theorem}, we have:

$$
\int_{\mathbb{R}} \big\vert\, \tilde{f}(t) - f(t) \,\big\vert d\nu(t) = \int_{\mathbb{R}} \big\vert\, \lim_k f_{N_k}(t) - f(t) \,\big\vert d\nu(t) = \lim_{k\in\mathbb{N}} \int_{\mathbb{R}} \big\vert\, f_{N_k}(t) - f(t) \,\big\vert d\nu(t) = \lim_k \| f_{N_k} - f \|_1 = 0
$$
which implies $\tilde{f} = f$, or $\big( f_{N_k} \big)_{k\in\mathbb{N}}$ also converges to $f$ $\nu$-almost everywhere.
    
\end{rem}

We are now ready to prove the precise description of those three types of type-$III$, under the assumption that $(X, R, \mu)$ is of type-$III$ and that infinitely many $X_n$ satisfies that $\vert\, X_n\,\vert > 2$. By our assumption that $\limsup_n \vert\, X_n\,\vert = \infty$, we must have $\mathcal{N}$ (see \textbf{Definition \ref{Definition 2.5}}) is an infinite set. For generality, we can also assume that $\mathbb{N} \backslash \mathcal{N}$ is infinite. According to (\ref{e13}), for every $i, j\in \mathcal{N} \backslash \{0\}$, the clustered points of both of the following sets will be in the associated flow:

$$
\bigcup_{i, j\in \mathcal{N}} \left\{ \log \frac{\mu_n(i)}{\mu_n(0)} + \log\frac{\mu_k(j)}{\mu_k(0)}: n\in\mathcal{I}_i, k\in \mathcal{I}_j \right\}, \hspace{1cm} \bigcup_{j\in \mathcal{N}} \bigcup_{i\notin \mathcal{N}} \left\{ \log \frac{\mu_n(j)}{\mu_n(0)} + \log \frac{\mu_m(i)}{\mu_m(0)}: n\in \mathcal{I}_j, m\in \mathcal{I}_i \right\}
$$
For each $i\in \mathcal{N}\backslash \{0\}$, define:

\begin{equation}\label{e15}
\mathcal{M}_i = \left\{ \frac{\mu_m(i)}{\mu_m(0)}: m\in \mathcal{I}_i \right\}
\end{equation}
Recall that we assume that for each $\mu_n$ in $\mu = \otimes_{n\in \mathbb{N}} \mu_n$, $\mu_n(0)$ is the maximum weight. Hence for all $i\in\mathcal{N}\backslash \{0\}$, $\mathcal{M}_i \subseteq (0, 1)$ and clustered points of $\mathcal{M}_i$ must be in $[0, 1]$.

\begin{theorem}\label{Theorem 2.8}

Given a type-$III$ Bernoulli scheme $G\curvearrowright (X, \mathcal{B}, \mu)$, if $\limsup_n \vert\, X_n\,\vert = \infty$, $\mu$ is of type-$III_1$ if and only if one of the following is true:

\begin{enumerate}[label = (\roman*)]

    \item the set of clustered points of the following set (presumably infinite)

    \begin{equation}\label{e16}
    \mathcal{M}_F = \bigcup_{i\notin \mathcal{N}} \left\{ \frac{\mu_m(i)}{\mu_m(0)}: m\in \mathcal{I}_i \right\}
    \end{equation}
    contains $0$ or there exists $i\in\mathcal{N}$ such that the multiplicative group generated by cluster points of $\mathcal{M}_i \cup \mathcal{M}_F$ is $(0, \infty)$.

    \item there exists two different $i, j\in\mathcal{N}$ such that $\mathcal{M}_i \cup \mathcal{M}_j$ either has $0$ as one of its clustered points, or all clustered points of $\mathcal{M}_i\cup \mathcal{M}_j$ are non-zero and the multiplicative group generated by clustered points of $\mathcal{M}_i \cup \mathcal{M}_j$ is $(0, \infty)$.

    \item for any $\epsilon\in(0, 1)$, there exists $i\in\mathcal{N}$ and $r_i$ a clustered point of $\mathcal{M}_i$ such that $\vert\,r_i\,\vert < \epsilon$. Equivalently:

    $$
    \inf_{i\in \mathcal{N}} \big( \liminf\mathcal{M}_i \big) =0
    $$
    where, by $\liminf \mathcal{M}_i$, each $\mathcal{M}_i$ is viewed as a sequence of numbers 
    
\end{enumerate}
    
\end{theorem}

\begin{proof}

We will first consider the case where $(ii)$ is true. The proof of this case will be based on (\ref{e13}). We will start from the case where there exists $r, s\in(0, 1)$, two index set $(m_n)_{n\in \mathbb{N}} \subseteq \mathcal{I}_i$, $(k_n)_{n\in \mathbb{N}} \subseteq \mathcal{I}_j$ such that:

\begin{equation}\label{e17}
\begin{aligned}
& \left\{ \frac{\mu_{m_n}(i)}{\mu_{m_n}(0)} \right\}_{n\in \mathbb{N}} \subseteq \mathcal{M}_i, \hspace{1cm} \lim_n \frac{\mu_{m_n}(i)}{\mu_{m_n}(0)} = r\\
& \left\{ \frac{\mu_{k_n}(j)}{\mu_{k_n}(0)} \right\}_{n\in \mathbb{N}} \subseteq \mathcal{M}_j, \hspace{1cm} \lim_n \frac{\mu_{k_n}(j)}{\mu_{k_n}(0)} = s\\
\end{aligned}
\end{equation}
and the multiplicative group generated by $r$ and $s$ is the entire $(0, \infty)$. Fix $(\epsilon_n)_{n\in \mathbb{N}} \subseteq (0, 1)$ and $F\in L^{\infty}\big( X\times \mathbb{R}, \mu\times \nu\big)^G$. According to \textbf{Remark \ref{Remark 2.8}} (in particular (\ref{e13})), whenever $n, N\in\mathbb{N}$ is large enough, for $\nu$-almost every $t\in \mathbb{R}$ and $\mu_N^*$-almost each $\bm{x}_N^*\in X_N^*$:

\begin{equation}\label{e48}
\begin{aligned}
& \hspace{1cm} \forall\, l\in X_N, \hspace{0.2cm} \left\vert\, F\left( (l, \bm{x}_N^*), t\pm o\left( \frac{2}{n} \right) \right) - F\left( (l, \bm{x}_N^* ), t+\log r+\log s \right) \,\right\vert < \epsilon_n\\
& \implies\, \left\vert\, F\left( \bm{x}, t\pm o\left( \frac{2}{n} \right) \right) - F\big( \bm{x}, t+\log r+\log s \big) \,\right\vert < \epsilon_n
\end{aligned}
\end{equation}
As (\ref{e48}) holds for all large enough $n\in\mathbb{N}$, we can then conclude that $\log rs$ is in the associated flow, or $rs$ is in the ratio set. If for all $i\in \mathcal{N}$, clustered points of $\mathcal{M}_i$ are non-zero and there exists two different $i, j\in \mathcal{N}$ such that the multiplicative group generated by clustered points of $\mathcal{M}_i\cup \mathcal{M}_j$ is $(0, \infty)$, there exists $r_i$, a clustered point of $\mathcal{M}_i$, and $r_j$ a clustered point of $\mathcal{M}_j$ such that $r_i$ and $r_j$ are rationally independent. Then fix $l\in \mathcal{N} \backslash \{i, j\}$ and let $r_l$ be a clustered point of $\mathcal{M}_l$. By the same reasoning, we will then have $r_lr_i$ and $r_lr_j$ in the ratio set. Since $r_i$, $r_j$ are rationally independent, so are $r_lr_i$ and $r_lr_j$, and hence the ratio set of $\mu$ is $(0, \infty)$. This proves that when the second condition in \textit{(ii)} is satisfied, $\mu$ is of type-$III_1$.\\

\noindent
Assume that the second condition of \textit{(i)} is satisfied. Then there exists non-zero two clustered points $r, s$ of $\mathcal{M}_i\cup \mathcal{M}_F$ such that $r$ and $s$ are rationally independent. By the same reasoning, we have both $rs$ and $r^2$ are in the ratio set. Since $r$ and $s$ are rationally independent, $r^2$ and $rs$ are rationally independent, and hence the ratio set of $\mu$ is of type-$III_1$.\\

\noindent
Next we will prove that $\mu$ is of type-$III_1$ whenever the first case in condition \textit{(ii)} is satisfied, and the same conclusion follows when the first case in condition \textit{(i)} is satisfied. Without confusion, we will pick $r$ and clustered point of $\mathcal{M}_i$ and $s$ a clustered point of $\mathcal{M}_j$ as in (\ref{e17}) but instead assume that $r=0$. According to (\ref{e12}), for an arbitrary $\epsilon\in (0, 1)$, whenever $n\in\mathbb{N}$ is large enough, we have:

$$
\left\vert\, F(\bm{x}, t) - F\left( \bm{x}, t+\log\frac{\mu_{m_n}(i)}{\mu_{m_n}(0)} + \log \frac{\mu_{k_n}(j)}{\mu_{k_n}(0)} \right) \,\right\vert < \epsilon
$$
for $\nu$-almost each $t\in\mathbb{N}$ and $\mu$-almost every $\bm{x}\in X$. In this case we cannot use the little-$o$ notation as $\log \dfrac{\mu_{m_n}(i)}{\mu_{m_n}(0)}$ will diverge to $-\infty$. As we are using notations introduced in (\ref{e17}) and $s\in [0, 1]$, it suffices to only consider:

\begin{equation}\label{e44}
\left\vert\, F(\bm{x}, t) - F\left( \bm{x}, t+\log \frac{\mu_{m_n}(i)}{\mu_{m_n}(0)} \right) \,\right\vert < \epsilon
\end{equation}
as the following sequence:

$$
\left\{ \log \frac{\mu_{m_n}(i)}{\mu_{m_n}(0)} + \log \frac{\mu_{k_n}(j)}{\mu_{k_n}(0)} \right\}_{n\in \mathbb{N}}
$$
will diverge to $-\infty$ regardless of the value of $s$. Fix $\bm{x}\in X$. By \textbf{Remark \ref{Remark 2.8}}, there exists a sequence of differentiable functions $(h_n)_{n\in \mathbb{N}}$ such that the sequence $(h_n)_{N\in \mathbb{N}}$ converges to $F(\bm{x}, \cdot)$ in $\|\cdot\|_1$ and pointwise $\nu$-almost everywhere. Since $F(\bm{x}, \cdot) \in L^{\infty}(\mathbb{R}, \nu)$, we can assume each $h_n$ has compact support, which will not cause conflict with our assumption that $(h_n)_{n\in \mathbb{N}}$ converges pointwise to $F(\bm{x}, \cdot)$ $\nu$-almost everywhere. Suppose that the support of each $h_n$ is contained in a closed finite interval $I_n$. We use $\vert\, I_n\,\vert$ to denote the length (or Lebesgue measure) of $I_n$.\\

\noindent
First, fix $(\delta_M) \subseteq (0, 1)$ such that $\delta_M \rightarrow 0$. Fix $M\in\mathbb{N}$ and $m'\in\mathbb{N}$ such that whenever $n\geq m'$, (\ref{e44}) holds for $\nu$-almost each $t\in \mathbb{R}$ and for $\epsilon=\delta_M$. Let $\mathcal{P}_M$ be a finite $\delta_M$-net of $[-M, M]$. Then let $N'(M)\in \mathbb{N}$ be such that:

\begin{equation}\label{e45}
\forall\,n'\geq N'(M), \hspace{0.5cm} \max_{t\in \mathcal{P}_M} \big\vert\, F(\bm{x}, t) - h_{n'}(t) \,\big\vert < \delta_M
\end{equation}
By our assumption, for all large enough $l\in\mathbb{N}$ with $l\geq m'$, we have:

\begin{equation}\label{e46}
\left\vert\, \log \frac{\mu_{m_l}(i)}{\mu_{m_l}(0)} \,\right\vert > 2\max\big( \vert\, I_{N'(M)} \,\vert\,,\, M \big)
\end{equation}
Fix $l\in\mathbb{N}$ so that (\ref{e46}) holds. Combining (\ref{e45}) and (\ref{e44}), we have:

$$
\forall\,t\in \mathcal{P}_M, \hspace{0.5cm} \left\vert\, h_{N'(M)}(t) - h_{N'(M)}\left( t\pm\log \frac{\mu_{m_l}(i)}{\mu_{m_l}(0)} \right) \,\right\vert < 2\delta_M
$$
According to (\ref{e46}), we have:

$$
\forall\,t\in\mathcal{P}_M, \hspace{0.3cm} t\pm \log \frac{\mu_{m_l}(i)}{\mu_{m_l}(0)} \notin I_{N'(M)} \hspace{1cm} \Longrightarrow \hspace{1cm} \forall\,t\in \mathcal{P}_M, \hspace{0.3cm} h_{N'(M)} \left( t\pm\log \frac{\mu_{m_l}(i)}{\mu_{m_l}(0)} \right) = 0
$$
which implies:

$$
\forall\,t\in\mathcal{P}_M, \hspace{0.3cm} \vert\, h_{N'(M)}(t) \,\vert < 2\delta_M
$$
Next, for the $M+1$ case, consider $m''\in\mathbb{N}$ such that whenever $n\geq m''$, (\ref{e44}) holds for $\nu$-almost every $t\in \mathbb{R}$ and for $\epsilon=\delta_{M+1}$. Next let $\mathcal{P}_{M+1}$ be a finite $\delta_{M+1}$-net that contains $\mathcal{P}_M$. Then we can find $N'(M+1)$ so that:

$$
\forall\,n'\geq N'(M+1), \hspace{0.5cm} \max_{t\in \mathcal{P}_{M+1}} \big\vert\, F(\bm{x}, t) - h_{n'}(t) \,\big\vert < \delta_{M+1}
$$
which is (\ref{e45}) but with $M$ replaced by $M+1$. Then find $l'\in\mathbb{N}$ such that (\ref{e46}) holds with $m_l$ replaced by $m_{l'}$ and $N'(M)$ replaced by $N'(M+1)$. As a result, we will have that:

$$
\forall\,t\in \mathcal{P}_{M+1}, \hspace{0.3cm} \vert\, h_{N'(M+1)}(t) \,\vert < 2\delta_{M+1}.
$$
For each $M\in\mathbb{N}$, put $f_M = h_{N'(M)}$. We then have $(f_M)_{M\in \mathbb{N}}$ converges to $F(\bm{x}, \cdot)$ in $\|\cdot\|_1$ and pointwise $\nu$-almost everywhere. Since $\delta_M \rightarrow 0$ and each $\mathcal{P}_M$ is a finite $\delta_M$-net of $[-M, M]$, for two fixed different points $t, t'\in [-M, M]$, we have:

$$
\vert\, f_{M'}(t) - f_{M'}(t') \,\vert \leq 4\delta_{M'} < \delta_M
$$
whenever $M'\geq M$ is large enough. Suppose that there exists $t, t'\in \mathbb{R}$ such that $F(\bm{x}, t) \neq F(\bm{x}, t')$. Define: 

$$
\delta = \big\vert\, F(\bm{x}, t) - F(\bm{x}, t') \,\big\vert
$$ 
Fix $M\in\mathbb{N}$ such that $t, t'\in [-M, M]$ and $\delta_M<3\delta$. Next fix $M'\in \mathbb{N}$ with $M'>M$ such that:

\begin{equation}\label{e47}
\big\vert\, f_{M'}(t) - F(\bm{x}, t) \,\big\vert < \frac{\delta}{3}, \hspace{1cm} \big\vert\, f_{M'}(t) - f_{M'}(t') \,\big\vert < \frac{\delta_M}{3}, \hspace{1cm} \big\vert\, f_{M'}(t') - F(\bm{x}, t') \,\big\vert < \frac{\delta}{3}
\end{equation}
With the choice of $M'$ and $M$ that make all inequalities in (\ref{e47}) hold, we then have:

$$
\begin{aligned}
\delta
& = \big\vert\, F(\bm{x}, t) - F(\bm{x}, t') \,\big\vert\\
& \leq \big\vert\, f_{M'}(t) - F(\bm{x}, t) \,\big\vert + \big\vert\, f_{M'}(t) - f_{M'}(t') \,\big\vert + \big\vert\, f_{M'}(t) - F(\bm{x}, t') \,\big\vert\\
& < \frac{\delta}{3} + \frac{\delta_M}{3} + \frac{\delta}{3} < \delta
\end{aligned}
$$
which is absurd. Therefore, we can now conclude $F(\bm{x}, \cdot)$ is a constant function. Since $\bm{x}\in X$ is arbitrarily fixed and $F\in L^{\infty}\big( X\times \mathbb{R}, \mu\times \nu\big)^G$ is also abitrarily fixed, we can then conclude the associated flow is trivial when $0$ is a clustered point of $\mathcal{M}_i$ for some $i\in\mathcal{N}$, and hence prove that $\mu$ is of type-$III_1$ when the first case in condition \textit{(ii)} is true. Similarly, when the first case in condition \textit{(i)} is true, $\mathcal{M}_F$ will contain a sequence converging to zero. Replacing $\left\{ \dfrac{\mu_{m_n}(i)}{\mu_{m_n}(0)} \right\}$ by that sequence in (\ref{e44}), then the rest will follow the same way.Similarly, if:

$$
\inf_{i\in \mathcal{N}} \big( \liminf \mathcal{M}_i \big)=0
$$
then there exists a sequence in $\bigcup_{i\in \mathcal{N}} \mathcal{M}_i$, which is converging to zero. Replacing $\left\{ \dfrac{\mu_{m_n}(i)}{\mu_{m_n}(0)} \right\}$ by that sequence in (\ref{e44}) and the rest will follow. We can now conclude that when any one of conditions \textit{(i)}, \textit{(ii)} and \textit{(iii)} hold, $\mu$ is of type-$III_1$.\\

\noindent
Conversely, suppose that the system $\big( X, R, \mu \big)$ is of type-$III_1$. Assume by contradiction that all of the following are true (namely, all of conditions $(i)$, $(ii)$ and $(iii)$ are false):

\begin{enumerate}[label = (\alph*)]

    \item the set of cluster points of $\mathcal{M}_F$ is included in $(0, 1]$, namely $\liminf\mathcal{M}_F > 0$ and for any $i\in\mathcal{N}$, the multiplicative group generated by clustered points of $\mathcal{M}_i \cup \mathcal{M}_F$ is either cyclic or trivial.

    \item for any two different $i, j \in\mathcal{N}$, $\liminf \mathcal{M}_i \cup \mathcal{M}_j > 0$ and the multiplicative group generated by clustered points of $\mathcal{M}_i\cup \mathcal{M}_j$ is either trivial or cyclic.

    \item there exists $\delta\in(0, 1)$ such that:

    $$
    \delta < \inf_{i\in\mathcal{N}} \big( \liminf\mathcal{M}_i \big)
    $$
    
\end{enumerate}

\noindent
According to condition $(c)$, the clustered points of $\mathcal{M}_F \cup \bigcup_{i\in \mathcal{N}} \mathcal{M}_i$ is finite and are all non-zero. By condition $(a)$ and $(b)$, any two different clustered points in $\mathcal{M}_F \cup \bigcup_{i\in \mathcal{N}} \mathcal{M}_i$ are rationally dependent. Therefore, there exists $\lambda\in (0, 1)$ such that the set of clustered points of $\mathcal{M}_F \cup \bigcup_{i\in \mathcal{N}} \mathcal{M}_i$ is a finite subset of $\big\{ \lambda^n \big\}_{n\in \mathbb{N}} \cup \{1\}$. Therefore, for any $K, M\in \mathbb{N}$ and $\epsilon\in (0, 1)$, for any $N\in\mathbb{N}$, $N\geq M$ that is large enough, we have:

$$
\begin{aligned}
& \hspace{0.4cm} \inf\left\{ \inf\left\{ \left\vert\, \lambda^n - \prod_{N+1 \leq k \leq N+K} \frac{\mu_k(x_k)}{\mu_k(y_k)} \,\right\vert: x_k, y_k\in X_k (N+1\leq k \leq N+K) \right\}: n\in \mathbb{Z} \right\}\\
& = \inf\left\{ \inf\left\{ \left\vert\, \lambda^n - \prod_{N+1 \leq k \leq N+K} \frac{\mu_k(x_k)}{\mu_k(0)} \left( \frac{\mu_k(y_k)}{\mu_k(0)} \right)^{-1} \,\right\vert: x_k, y_k\in X_k (N+1\leq k \leq N+K) \right\}: n\in \mathbb{Z} \right\} < \epsilon
\end{aligned}
$$
According to the description of the Radon-Nikodym cocycle in the proof of \textbf{Lemma \ref{Lemma 2.1}} and \textbf{Lemma \ref{Lemma 2.2}}, we can then conclude the essential range of the Radon-Nikodym derivative is $\big\{ \lambda^n \big\}_{n\in \mathbb{Z}} \cup \{1\}$, which contradicts our assumption that $\mu$ is of type-$III_1$.
    
\end{proof}

\begin{theorem}\label{Theorem 2.11}

Given a type-$III$ Bernoulli scheme $G\curvearrowright (X, \mathcal{B}, \mu)$, if $\limsup_n \vert\, X_n\,\vert = \infty$, $\mu$ is of type-$III_{\lambda}$ for some $\lambda\in (0, 1)$ if and only if the set of cluster points of $\mathcal{M}_F \cup \bigcup_{i\in \mathcal{N} \backslash \{0\}} \mathcal{M}_i$ is a finite subset of $\big\{ \lambda^n \big\}_{n\in \mathbb{N}} \cup\{1\}$, and the multiplicative group generated by that set of clustered points is $\big\{ \lambda^n \big\}_{n\in \mathbb{Z}}$ (definitions of each $\mathcal{M}_i$, $i\in\mathcal{N}$ and $\mathcal{M}_F$ can be found in (\ref{e15}) and (\ref{e16})).
    
\end{theorem}

\begin{proof}

Let $\mathcal{F}$ be the set of clustered points of $\mathcal{M}_F \cup \bigcup_{i\in \mathcal{N} \backslash \{0\}} \mathcal{M}_i$. First assume $\mathcal{F}$ is a finite subset of $\big\{ \lambda^n \big\}_{n\in \mathbb{N}} \cup\{1\}$ and multiplicative group generated by $\mathcal{F}$ is $\big\{ \lambda^n \big\}_{n\in \mathbb{Z}}$. Then there exists $r, s\in \mathcal{F}$ such that $rs^{-1} = \lambda$. Similarly by (\ref{e48}) we have both $r^2$ and $rs$ are in the ratio set, and hence $\lambda$ is in the ratio set. Also, since $\mathcal{F}$ is a finite subset of $\big\{ \lambda^n \big\}_{n\in \mathbb{N}} \cup \{1\}$, for any $\bm{x}, \bm{y}\in X$, for any $\epsilon\in (0, 1)$ and for any $M\in \mathbb{N}$, there exists $K\in \mathbb{N}$ and whenever $N\geq M$ that is large enough, we have:

\begin{equation}\label{e18}
\begin{aligned}
& \hspace{0.4cm} \inf\left\{ \inf\left\{ \left\vert\, \lambda^n - \prod_{N+1 \leq k \leq N+K} \frac{\mu_k(x_k)}{\mu_k(y_k)} \,\right\vert: x_k, y_k\in X_k (N+1\leq k \leq N+K) \right\}: n\in \mathbb{Z} \right\}\\
& = \inf\left\{ \inf\left\{ \left\vert\, \lambda^n - \prod_{N+1 \leq k \leq N+K} \frac{\mu_k(x_k)}{\mu_k(0)} \left( \frac{\mu_k(y_k)}{\mu_k(0)} \right)^{-1} \,\right\vert: x_k, y_k\in X_k (N+1\leq k \leq N+K) \right\}: n\in \mathbb{Z} \right\}\\
& = \min\left\{ \inf\left\{ \left\vert\, s - \prod_{N+1 \leq k \leq N+K} \frac{\mu_k(x_k)}{\mu_k(0)} \left( \frac{\mu_k(y_k)}{\mu_k(0)} \right)^{-1} \,\right\vert: x_k, y_k\in X_k (N+1\leq k \leq N+K) \right\}: s\in \mathcal{F} \right\} < \epsilon
\end{aligned}
\end{equation}
Assume by contradiction that there exists $r\in(0, 1)$ and $\delta\in (0, 1)$ such that $r$ is in the ratio set and:

$$
\delta = \inf_{n\in \mathbb{N} \cup \{0\}} \big\vert\, \lambda^n - r \,\big\vert
$$
According to the description of the Radon-Nikodym cocycle in the proof of \textbf{Lemma \ref{Lemma 2.1}} and \textbf{Lemma \ref{Lemma 2.2}}, by letting $\epsilon$ be $\delta$ in (\ref{e18}), our assumption that $r$ can be approached by values of the Radon-Nikodym cocycle contradicts (\ref{e18}). Therefore, the ratio set of $\mu$ is of type-$III_{\lambda}$, or $\mu$ is of type-$III_{\lambda}$.\\ 

\noindent
Conversely, suppose $\big(X, R, \mu\big)$ is of type-$III_{\lambda}$ for some $\lambda\in (0, 1)$. By the same reasoning based on (\ref{e18}), we have $\mathcal{F}$ is necessarily a subset of $\big\{ \lambda^n \big\}_{n\in \mathbb{N}} \cup \{1\}$, or the ratio set will contain other numbers outside of $\big\{ \lambda^n \big\}_{n\in \mathbb{Z}}$. Since $\lambda\in (0, 1)$, if $\mathcal{F}$ is infinite, we will have:

$$
\inf_{i\in \mathcal{N} \backslash \{0\}} \big( \liminf (\mathcal{M}_i\cup \mathcal{M}_F) \big) = 0
$$
According to \textbf{Theorem \ref{Theorem 2.8}}, this will imply that $\mu$ is of type-$III_1$. Therefore $\mathcal{F}$ must be a finite subset of $\big\{ \lambda^n \big\}_{n\in \mathbb{N}} \cup \{1\}$. Therefore, for any $\epsilon\in (0, 1)$ and for any $M\in \mathbb{N}$, there exists $K\in \mathbb{N}$ such that whenever $N\geq M$ is large enough, we have (\ref{e18}) holds. If $\lambda\in \mathcal{F}$ then clearly the multiplicative group generated by $\mathcal{F}$ is $\big\{ \lambda^n \big\}_{n\in \mathbb{Z}}$. Otherwise, assume by contradiction that the multiplicative group generated by $\mathcal{F}$ is $\big\{ \lambda^{kn} \big\}_{n\in \mathbb{Z}}$ for some integer $k\geq 2$. In this case we must have $\mathcal{F} \subseteq \big\{ \lambda^{kn} \big\}_{n\in \mathbb{N}} \cup \{1\}$. Put:

$$
\delta' = \vert\, \lambda^k - \lambda \,\vert = \inf_{n\in \mathbb{N} \cup \{0\}} \big\vert\, \lambda^{kn} - \lambda \,\big\vert
$$
According to the description of the Radon-Nikodym cocycle in the proof of \textbf{Lemma 2.3} and \textbf{Lemma 2.4}, since $\lambda$ is in the ratio set, for any $M\in\mathbb{N}$, there exists $K\in\mathbb{N}$ and whenever $N\geq M$ is large enough, we have:

$$
\min\left\{ \left\vert\, \lambda - \prod_{N+1\leq k \leq N+K} \frac{\mu_k(x_k)}{\mu_k(0)} \left( \frac{\mu_k(y_k)}{\mu_k(0)} \right)^{-1} \,\right\vert: x_k, y_k\in X_k (N+1\leq k \leq N+K) \right\} < \delta'
$$
which is absurd. Therefore, we must have the multiplicative group generated by $\mathcal{F}$ is $\big\{ \lambda^n \big\}_{n\in \mathbb{Z}}$.
    
\end{proof}

\begin{theorem}\label{Theorem 2.12}

Given a type-$III$ Bernoulli scheme $G\curvearrowright (X, \mathcal{B}, \mu)$, if $\limsup_n \vert\, X_n\,\vert = \infty$, $\mu$ is of type-$III_0$ if and only if the set of clustered points of $\mathcal{M}_F \cup \bigcup_{i\in \mathcal{N}\backslash \{0\}} \mathcal{M}_i$ is $\{1\}$.
    
\end{theorem}

\begin{proof}

The conclusion follows immediately by \textbf{Theorem \ref{Theorem 2.8}} and \textbf{Theorem \ref{Theorem 2.11}}.
    
\end{proof}

\begin{rem}

We will now show that descriptions in \textbf{Theorem \ref{Theorem 2.8}, \ref{Theorem 2.11} and \ref{Theorem 2.12}} are independent to permutation in each $X_n$ as each $X_n$ at the very beginning is assumed to be a discrete countable set. For each $n\in\mathbb{N}$ let $\pi_n$ be a permutation of $X_n$. Without losing generality we can assume that $\pi_n(0)=0$ so that $\mu_n\big( \pi_n(0) \big)$ has the biggest weight in $X_n$. It suffices to show that both (\ref{e12}) (or (\ref{e13})) are independent to permutation.\\

\noindent
Recall that in \textbf{Definition \ref{Definition 2.5}} we define $\mathcal{N}$ and $\mathcal{I}_i$ for each $i\in \mathcal{N}$. For each $i\in \mathcal{N}$ define:

$$
\mathcal{J}_i = \big\{ \pi_n^{-1}(i): n\in \mathcal{I}_i \big\}
$$
Since $\limsup_n \vert\, X_n\,\vert =\infty$ and $\mathcal{N}$ is infinite, the following set is infinite:

$$
\big\{ i\in\mathcal{N}: \mathcal{J}_i\,\text{  infinite  } \big\}
$$
Fix two different $i, j\in \mathcal{N} \backslash \{0\}$ such that both $\mathcal{J}_i$ and $\mathcal{J}_j$ are infinite. Then there exists $(i_m)_{m\in \mathbb{N}} \subseteq \mathcal{I}_i$, $(j_k)_{k\in \mathbb{N}} \subseteq \mathcal{I}_j$ such that for each $m\in \mathcal{I}_i$, $\pi_m(i_m) = i$ and for each $k\in \mathcal{I}_j$, $\pi_k(j_k) = j$. Then similar to how we deduce (\ref{e12}), for each $\epsilon\in (0, 1)$ and each $m\in \mathcal{I}_i$, $k\in \mathcal{I}_j$, whenever $N\in\mathbb{N}$ is large enough:

$$
\left\vert\, F\big( (0, \bm{x}_N^*), t \big) - F\left( ( 0, \bm{x}_N^*), t+ \log\frac{\mu_m\big( \pi_m(i_m) \big)}{\mu_m (0)} - \log \frac{\mu_k\big( \pi_k(j_k) \big)}{\mu_k(0)} \right) \,\right\vert < \epsilon
$$
for $\nu$-almost each $t\in \mathbb{R}$ and $\mu_N^*$-almost each $\bm{x}_N^* \in X_N^*$. Then similar to how we deduced (\ref{e13}), we then have:

$$
\left\vert\, F\big( (l, \bm{x}_N^*), t \big) - F\left( ( l, \bm{x}_N^*), t+ \log\frac{\mu_m\big( \pi_m(i_m) \big)}{\mu_m (0)} - \log \frac{\mu_k\big( \pi_k(j_k) \big)}{\mu_k(0)} \right) \,\right\vert < \epsilon
$$
for each $l\in X_N$. When we permute each $X_n$ by $\pi_n$ for each $n\in \mathbb{N}$, instead of looking at clustered points of $\left\{ \dfrac{\mu_m(i)}{\mu_m(0)} \right\}_{m\in \mathcal{I}_i}$ for $i\in \mathcal{N}$, we consider $\left\{ \dfrac{\mu\big( \pi_m(i_m) \big)}{\mu_m(0)} \right\}_{m\in \mathcal{I}_i}$ where $(i_m)_{m\in \mathbb{N}} \subseteq \mathcal{I}_i$ and $\pi_m(i_m) = i$ for all $m\in \mathbb{N}$, and this will not affect the proof \textbf{Theorem \ref{Theorem 2.8}, \ref{Theorem 2.11}, \ref{Theorem 2.12}}. Therefore, \textbf{Theorem \ref{Theorem 2.8}, \ref{Theorem 2.11}, \ref{Theorem 2.12}} can also be applied to a general Bernoulli scheme.
    
\end{rem}

\begin{cor}

Given a type-$III$ ITPFI factor $\mathcal{M} = R(\mathcal{M}_n, v_n)$ where each $\mathcal{M}_n$ is an $I_{k_n}$ factor for some $2\leq k_n \leq \infty$. Suppose $\limsup_n k_n = \infty$ and for each $n\in\mathbb{N}$, $\lambda_{n, 1}$ is the maximum in $\operatorname{Sp}\big( v(n) \slash \mathcal{M}(n) \big)$. Define:

$$
\mathcal{F}_n = \left\{ \frac{\lambda_{n, i}}{\lambda_{n, 1}}:1\leq i \leq k_n \right\}
$$
and let $\Lambda$ be the set of clustered points in $\bigcup_{n\in \mathbb{N}} \mathcal{F}_n$. Then $\mathcal{M}$ is of:

\begin{enumerate}[label = (\arabic*)]

    \item type-$III_1$ if and only if $0\in \Lambda$ or the multiplicative group generated by $\Lambda\backslash \{0\}$ is $(0, \infty)$.
    \item type-$III_{\lambda}$ for some $\lambda\in (0, 1)$ if and only if $\lambda\in \Lambda$ and $\Lambda\backslash \{1\}$ is a finite subset of $(\lambda^n)_{n\in \mathbb{N}}$. 
    \item type-$III_0$ if and only if $\Lambda=\{1\}$.
    
\end{enumerate}
    
\end{cor}

\begin{proof}

The conclusion follows immediately by the correspondence between an ITPFI factor and its associated Bernoulli scheme, \textbf{Proposition \ref{Proposition 2.1}}, \textbf{Theorem \ref{Theorem 2.8}, \ref{Theorem 2.11}, \ref{Theorem 2.12}}.
    
\end{proof}

\subsubsection{When $\limsup_n\vert\, X_n\,\vert < \infty$}
$\hspace{1cm}\\$

According to \textbf{Lemma \ref{Lemma 2.1}} and \textbf{Lemma \ref{Lemma 2.2}}, the ratio set of a type-$III$ Bernoulli scheme $G \curvearrowright \big(X, R, \mu\big)$  is determined by the multiplicative group generated by the clustered points of fractions of weights of $\mu_n$. Therefore, we can remove finitely many $X_n$ from $X=\prod_{n\in \mathbb{N}}X_n$ without changing its type. According to \cite{41}, an ITPFI factor $\mathcal{M} = R(\mathcal{M}_n, v_n)$ where $\limsup_n k_n<\infty$ is isomorphic to an $\textbf{ITPFI}_2$ factor (see \textbf{Definition 1.3}). From now on it suffices to only consider the case where $X=\{0, 1\}^{\mathbb{N}}$. We only consider the case where, given that $\mu = \otimes_{n\in \mathbb{N}} \mu_n$, $\mu_n(0) > \mu_n(1)$ for all $n\in\mathbb{N}$. Then there exists a sequence $(\lambda_n)_{n\in \mathbb{N}} \subseteq (0, 1)$ such that for each $n\in \mathbb{N}$:

\begin{equation}\label{e3}
\mu_n(0) = \frac{1}{1+\lambda_n}, \hspace{1cm} \mu_n(1) = \frac{\lambda_n}{1+\lambda_n}
\end{equation}

\begin{lem}[{\cite[Lemma 1.3]{15}}]\label{Lemma 2.13}

In $G\curvearrowright (X, \mu)$ where $X=\{0, 1\}^{\mathbb{N}}$ and $\mu$ is described by a sequence $(\lambda_n)_{n\in \mathbb{N}} \subseteq (0, 1)$ as in (\ref{e3}), suppose for some $\lambda\in (0, 1]$, $\lim_n \lambda_n=\lambda$. Then:

\begin{itemize}

    \item if $(\lambda_n)_{n\in \mathbb{N}}$ is summable, then $\mu$ is equivalent to the following measure on $\{0, 1\}^{\mathbb{N}}$:

    $$
    \bigotimes_{n\in\mathbb{N}} \left(\frac{1}{1+\lambda}\,,\, \frac{\lambda}{1+\lambda} \right)
    $$

    and hence $\mu$ is of type-$III_{\lambda}$ when $\lambda\neq 1$, or of type-$II_1$ when $\lambda=1$.

    \item if $(\lambda_n)_{n\in\mathbb{N}}$ is not summable, then $\mu$ is of type-$III_1$.
    
\end{itemize}
    
\end{lem}

\begin{theorem}[\cite{12}, \cite{45}]\label{Theorem 2.16}

In general, given two ergodic non-singular group action $G_i \curvearrowright (X_i, \mu_i)$ ($i=1, 2$), each of which is on a standard measure space, consider the product group action $(G_1\times G_2) \curvearrowright \big( X_1\times X_2, \mu_1 \times\mu_2 \big)$:

\begin{enumerate}[label = (\arabic*)]

    \item if $\mu_1$ ($\mu_2$ resp.) is of type-$II_1$, $\mu_1\times \mu_2$ is of same type of $\mu_2$ ($\mu_1$ resp.).

    \item if $\mu_1$ or $\mu_2$ is of type-$III_1$, then so is $\mu_1 \times \mu_2$.

    \item if for some $r, s\in (0, 1)$ such that $\mu_1$ is of type-$III_r$ and $\mu_2$ is of type-$III_s$, then $\mu_1\times \mu_2$ is of type-$III_1$ if the multiplicative group generated by $r$ and $s$ is $(0, \infty)$; otherwise, $\mu_1 \times \mu_2$ is of type-$III_{\lambda}$ for some $\lambda\in (0, 1)$ such that the multiplicative group generated by $r$ and $s$ is $\big( \lambda^n \big)_{n\in \mathbb{Z}}$.
    
\end{enumerate}
    
\end{theorem}

First we fix a type-$III$ Bernoulli scheme $G\curvearrowright (X, \mathcal{B}, \mu)$ where $X=\{0, 1\}^{\mathbb{N}}$ and $\mu = \otimes_{n\in \mathbb{N}} \mu_n$ where $(\mu_n)_{n\in \mathbb{N}}$ is defined by $(\lambda_n)_{n\in \mathbb{N}} \subseteq (0, 1)$ as in (\ref{e3}). Define $\Lambda$ to be the set of clustered points of $(\lambda_n)_{n\in \mathbb{N}}$, and clearly $\Lambda\subseteq [0, 1]$. For each $\lambda\in\Lambda$, define $\mathcal{N}(\lambda) \subseteq \mathbb{N}$ to be an infinite subset such that:

$$
\lim_{n\in \mathcal{N}_{\lambda}} \lambda_n = \lambda
$$
If there exist two different $\lambda, \lambda'\in \Lambda$, the intersection of $\mathcal{N}(\lambda)$ and $\mathcal{N}(\lambda')$ is at most finite. With the sequence $(\lambda_n)_{n\in \mathbb{N}}$ that defines each $\mu_n$ (and hence $\mu$) as in (\ref{e3}), there exists $(\epsilon_n)_{n\in \mathbb{N}} \subseteq (0, 1)$ such that $\lim_n \epsilon_n = 0$ and:

\begin{equation}\label{e49}
\forall\,\lambda\in \Lambda\, \forall\, n\in\mathbb{N}, \hspace{0.3cm}
\lambda_n = 
\begin{cases}
\lambda e^{-\epsilon_n}, \hspace{1cm} \lambda\neq 0,\, n\in \mathcal{N}(\lambda)\\
1 - e^{-\epsilon_n}, \hspace{0.6cm} \lambda=0,\, n\in\mathcal{N}(0)
\end{cases}
\end{equation}

\begin{prop}\label{Proposition 2.13}

Given a type-$III$ Bernoulli scheme $G\curvearrowright (X, \mathcal{B}, \mu)$ where $X=\{0, 1\}^{\mathbb{N}}$ and $\mu = \otimes_{n\in \mathbb{N}} \mu_n$ is defined by (\ref{e3}) and (\ref{e49}), if $\Lambda = \{0, 1\}$, then $\mu$ is of type-$III_0$.
    
\end{prop}

\begin{proof}

Let $D$ denote the Radon-Nikodym cocycle as in the proof of \textbf{Lemma \ref{Lemma 2.1}}. Under our assumption, for any $\bm{x}, \bm{y}\in X$, observe that:

\begin{equation}\label{e50}
\begin{aligned}
\log D(\bm{x}, \bm{y})
& = \left( \sum_{n\in \mathcal{N}(0)} (-1)^{x_n-y_n} \big( -\epsilon_n - \log (1-e^{-\epsilon_n}) \big) \right) + \left( \sum_{n\in \mathcal{N}(1)} (-1)^{x_n-y_n}\big( -\epsilon_n + \log\lambda \big) \right)\\
& = \left( \sum_{n\in\mathcal{N}(0)} (-1)^{x_n-y_n} \Big\vert\, \log\big( 1-e^{-\epsilon_n} \big) \,\Big\vert \right) + \left( \sum_{n\in \mathcal{N}(0) \cup \mathcal{N}(1)} (-1)^{y_n-x_n}\epsilon_n \right)
\end{aligned}
\end{equation}
Indeed since $\big( \bm{x}, \bm{y} \big)\in R$, $\log D(\bm{x}, \bm{y})$ is a finite sum. Fix $\bm{x}\in X$, $\epsilon\in (0, 1)$ and $M\in\mathbb{N}$. To prove that $\mu$ is of type-$III_0$, according to \textbf{Lemma \ref{Lemma 2.1}}, we need to find the desired $\bm{y}\in X$, $N, K\in\mathbb{N}$ with $N\geq M$ such that $x_n=y_n$ for all $n\notin \{N+1, N+2, \cdots, N+K\}$ and:

$$
\log D(\bm{x}, \bm{y}) \in (-\epsilon, \epsilon) \hspace{1cm} \text{  or  }\hspace{1cm} \left\vert\, \log D(\bm{x}, \bm{y}) \,\right\vert > \frac{1}{\epsilon} 
$$
First fix $M\in\mathbb{N}$ and $\epsilon\in (0, 1)$. We can then find $\delta\in (0, 1)$ such that $\big\vert\, \log\delta \,\big\vert > \epsilon^{-1}$. Next fix $N_1\in\mathbb{N}$ such that:

$$
\forall\,n\geq N_1, \hspace{0.2cm} \Big\vert\, \log\big( 1-e^{-\epsilon_n} \big) \,\Big\vert > \big\vert\, \log\delta \,\big\vert
$$
Since $\Lambda = \{0, 1\}$, $\mathcal{N}(0)$ must be infinite and we either have $N_1\in \mathcal{N}(0)$, or can replace $N_1$ by a larger integer from $\mathcal{N}(0)$. We can further assume that $N_1$ is large enough so that $\epsilon_n$ is negligible compared to $\epsilon^{-1}$. Then clearly there exists $\bm{y}\in X$ such that $y_n=x_n$ for all $n\in \mathbb{N} \backslash \{N_1\}$ and $x_{N_1} \neq y_{N_1}$. In this case, we have $\big\vert\, \log D(\bm{x}, \bm{y}) \,\big\vert > \epsilon^{-1}$.
 
\end{proof}

\begin{theorem}\label{Theorem 2.17}

In the set-up of \textbf{Proposition \ref{Proposition 2.13}}, $\mu$ is of type-$III_{\lambda}$ for some $\lambda\in (0, 1)$ if and only if all of the following are true:

\begin{enumerate}[label = (\arabic*)]

    \item the multiplicative group generated by $\Lambda$ is $\big( \lambda^n \big)_{n\in \mathbb{Z}}$.

    \item for each $s \in \Lambda$, the sequence $\big( \epsilon_n \big)_{n\in \mathcal{N}(s)}$ is summable.
    
\end{enumerate}
    
\end{theorem}

\begin{proof}

First we assume that $\mu$ is of type-$III_{\lambda}$ for some $\lambda\in (0, 1)$. Let $D$ denote the Radon-Nikodym cocycle. If there exists $r\in (0, 1)$ such that $r\in \Lambda$ and $r\notin \{1\}\cup \big( \lambda^n \big)_{n\in \mathbb{N}}$, according to (\ref{e50}) and the reasoning in the proof of \textbf{Proposition \ref{Proposition 2.13}}, $\log r$ is in the essential range of $\log D$, and hence $r$ is in the ratio set of $\mu$, which is absurd. Therefore, $\Lambda$ must be a subset of $\{1\} \cup \big\{ \lambda^n \big\}_{n\in \mathbb{N}}$. For each $t\in \Lambda$, observe that:

\begin{equation}\label{e51}
\left( \{0, 1\}^{\mathbb{N}}, \mu \right) = \left( \prod_{n\in \mathcal{N}(t)}\{0, 1\}, \bigotimes_{n\in \mathcal{N}(t)}\mu_n \right) \bigtimes \left( \prod_{n\in \mathbb{N} \backslash \mathcal{N}(t)}\{0, 1\}, \bigotimes_{n\in \mathbb{N} \backslash \mathcal{N}(t)}\mu_n \right)
\end{equation}
where we use $\mathcal{M}$, $\mathcal{M}_t$, $\mathcal{M}^t$ to denote the ITPFI factors, each of which is respectively associated with the first, second and third Bernoulli scheme in (\ref{e51}). If there exists $t\in \Lambda$ such that $\big( \epsilon_n \big)_{n\in \mathcal{N}(t)}$ is not summable, then $\mathcal{M}_t$ is of type-$III_1$ according to \textbf{Lemma \ref{Lemma 2.13}}. Then $\mathcal{M}$ is of type-$III_1$ by \textbf{Theorem \ref{Theorem 2.16}}, which contradicts our assumption. We can now assume for all $t\in \Lambda$, $\big( \epsilon_n \big)_{n\in \mathcal{N}(t)}$ is summable, and hence $\mathcal{M}_t \cong R_t$ according to \textbf{Lemma \ref{Lemma 2.13}}, where $R_t$ is defined in \textbf{Definition \ref{Definition 1.11}}. In this case we have:

\begin{equation}\label{e52}
\mathcal{M} \cong \bigotimes_{t\in\Lambda}R_t
\end{equation}
Then the asymptotic ratio set of $\mathcal{M}$ is the closure of the multiplicative group generated by $\Lambda$. According to \textbf{Proposition \ref{Proposition 2.1}}, we must have $\lambda\in \Lambda$ as we assume $\mu$ is of type-$III_{\lambda}$. Conversely, when both conditions are true, we can rewrite $\mathcal{M}$ as in (\ref{e52}). Then the conclusion follows by \textbf{Proposition \ref{Proposition 2.1}}.

\end{proof}

\begin{theorem}\label{Theorem 2.18}

In the set-up of \textbf{Proposition \ref{Proposition 2.13}}, $\mu$ is of type-$III_1$ if and only if one of the following is true:

\begin{itemize}
    \item there exists $t\in \Lambda \backslash \{0\}$ such that $\big( \epsilon_n \big)_{n\in \mathcal{N}(t)}$ is not summable.
    \item for each $s \in \Lambda \backslash \{0\}$, $\big( \epsilon_n \big)_{n\in \mathcal{N}(s)}$ is summable and the multiplicative group generated by $\Lambda$ is $(0, \infty)$.
\end{itemize}
    
\end{theorem}

\begin{proof}

First we let $\mathcal{M}$ denote the ITPFI factor associated with the given Bernoulli scheme. For any $t\in \Lambda$, we have $\mathcal{M} \cong \mathcal{M}_t \otimes \mathcal{M}^t$ where $\mathcal{M}_t$ is the ITPFI factor associated with the first Bernoulli scheme in (\ref{e51}) and $\mathcal{M}^t$ is associated with the second one in (\ref{e51}). If, for some $s\in \Lambda$, $\big( \epsilon_n \big)_{n\in \mathcal{N}(s)}$ is not summable, according to \textbf{Lemma \ref{Lemma 2.13}} and \textbf{Theorem \ref{Theorem 2.16}}, $\mathcal{M}$ is of type-$III_1$. If the second condition holds instead, then for all $s\in \Lambda$, $\mathcal{M}_s \cong R_s$ according to \textbf{Lemma \ref{Lemma 2.13}}. Then for all $s\in \Lambda$:

$$
\mathcal{M}\otimes R_s \cong \big( \mathcal{M}^s \otimes R_s \big) \otimes R_s \cong \mathcal{M}^s \otimes \big( R_s \otimes R_s \big) \cong \mathcal{M}
$$
which implies that $\Lambda$ is included in the asymptotic ratio set of $\mathcal{M}$, and hence the ratio set of $\mu$ is $[0, \infty)$ according to \textbf{Proposition \ref{Proposition 2.1}}.\\

\noindent
Conversely, assume that $\mu$ is of type-$III_1$. Again by \textbf{Proposition \ref{Proposition 2.1}} the asymptotic ratio set of $\mathcal{M}$ is $[0, \infty)$. Suppose that for all $s\in \Lambda$, $\big( \epsilon_n \big)_{n\in \mathcal{N}(s)}$ is summable. Then we can rewrite $\mathcal{M}$ as in (\ref{e52}), and hence the asymptotic ratio set of $\mathcal{M}$ is the closure of the multiplicative group generated by $\Lambda$. We then can conclude that the multiplicative group generated by $\Lambda$ is $(0, \infty)$.
    
\end{proof}

\begin{cor}

Suppose $\mathcal{M} = R(\mathcal{M}_n, v_n)$ is a type-$III$ $\text{ITPFI}_2$ factor where for each $n\in\mathbb{N}$:

$$
\operatorname{Sp}\Big( v(n) \slash \mathcal{M}(n) \Big) = \left\{ \frac{1}{1+\lambda_n} \,,\, \frac{\lambda_n}{1+\lambda_n} \right\}
$$
and the sequence $(\lambda_n)_{n\in \mathbb{N}}$ is defined by a vanishing sequence $(\epsilon_n)_{n\in \mathbb{N}} \subseteq (0, 1)$ as in (\ref{e49}). Let $\Lambda$ denote the set of clustered points of $(\lambda_n)_{n\in \mathbb{N}}$ and for each $\lambda\in\Lambda$, let $\mathcal{N}(\lambda) \subseteq \mathbb{N}$ be such that $\lim_{n\in \mathcal{N}(\lambda)} \lambda_n = \lambda$. Then $\mathcal{M}$ is of:

\begin{enumerate}[label = (\arabic*)]

    \item type-$III_1$ if one of the following is true:

    \begin{itemize}
        \item there exists $t\in \Lambda\backslash \{0\}$ such that $\big( \epsilon_n \big)_{n\in \mathcal{N}(t)}$ is not summable.
        \item for each $s\in\Lambda\backslash \{0\}$, $\big( \epsilon_n \big)_{n\in \mathcal{N}(s)}$ is summable and the multiplicative group generated by $\Lambda$ is $(0, \infty)$.
    \end{itemize}

    \item type-$III_{\lambda}$ for some $\lambda\in (0, 1)$ if and only if for each $s\in\Lambda\backslash \{0\}$, $\big( \epsilon_n \big)_{n\in \mathcal{N}(s)}$ is summable and the multiplicative group generated by $\Lambda$ is $\big( \lambda^n \big)_{n\in \mathbb{Z}}$.

    \item type-$III_0$ if and only if $\Lambda = \{0, 1\}$.
    
\end{enumerate}
    
\end{cor}

\bibliographystyle{amsalpha}
\bibliography{Reference_List_1}

\end{document}